\pdfoutput=1

\documentclass[reqno,a4paper]{amsart}
\usepackage[utf8]{inputenc}
\usepackage[T1]{fontenc}
\usepackage[english]{babel}
\usepackage{amssymb}
\usepackage[foot]{amsaddr}

\usepackage{graphicx}
\usepackage{wrapfig}
\usepackage{subfig}
\usepackage{enumerate}
\usepackage[usenames,dvipsnames]{color}
\usepackage{dsfont}
\usepackage{microtype}

\usepackage{mathtools}
\mathtoolsset{showonlyrefs,showmanualtags}

\newtheorem{thm}{Theorem}[section]
\newtheorem{cor}[thm]{Corollary}
\newtheorem{lem}[thm]{Lemma}
\newtheorem{prop}[thm]{Proposition}
\newtheorem{assume}[thm]{Assumption}
\theoremstyle{definition}
\newtheorem{defn}[thm]{Definition}

\theoremstyle{remark}
\newtheorem{rem}[thm]{Remark}

\numberwithin{equation}{section}

\DeclarePairedDelimiter{\abs}{\lvert}{\rvert}
\DeclarePairedDelimiter{\norm}{\lVert}{\rVert}
\DeclarePairedDelimiter{\bra}{(}{)}

\DeclarePairedDelimiter{\set}{\{}{\}}

\DeclareMathAlphabet{\mathup}{OT1}{\familydefault}{m}{n}
\newcommand{\dx}[1]{\mathop{}\!\mathup{d} #1}
\newcommand{\pderiv}[3][]{\frac{\mathop{}\!\mathup{d}^{#1} #2}{\mathop{}\!\mathup{d} #3^{#1}}}

\newcommand{\eps}{\varepsilon}

\DeclareMathOperator*{\argmin}{arg\,min}

\DeclareMathOperator*{\supp}{supp}

\DeclareMathOperator{\Id}{Id}
\DeclareMathOperator{\loc}{loc}
\DeclareMathOperator{\Var}{Var}
\DeclareMathOperator{\LSI}{LSI}

\DeclareMathOperator{\convex}{convex}
\DeclareMathOperator{\unimodal}{unimodal}
\DeclareMathOperator{\Kramers}{Kramers}

\newcommand{\N}{\mathds{N}}
\newcommand{\R}{\mathds{R}}

\newcommand{\cC}{\ensuremath{\mathcal C}}
\newcommand{\cD}{\ensuremath{\mathcal D}}

\newcommand{\cF}{\ensuremath{\mathcal F}}

\newcommand{\cH}{\ensuremath{\mathcal H}}

\newcommand{\cM}{\ensuremath{\mathcal M}}
\newcommand{\ctM}{\ensuremath{\widetilde{\mathcal{M}}}}
\newcommand{\ctF}{\ensuremath{\widetilde{\mathcal{F}}}}

\newcommand{\cP}{\ensuremath{\mathcal P}}

\newcommand{\cT}{\ensuremath{\mathcal T}}

\usepackage{hyperref}
\definecolor{darkblue}{rgb}{0,0,0.6}
\hypersetup{
    pdftitle={A constrained Fokker-Planck equation as gradient flow and  its longtime behaviour},    
    pdfauthor={Simon Eberle, Barbara Niethammer, André Schlichting},
    colorlinks=true,       
    linkcolor=darkblue,          
    citecolor=darkblue,        
    filecolor=darkblue,      
    urlcolor=darkblue           
}

\title[A constrained Fokker-Planck equation: gradient flow formulation]{Gradient flow formulation and longtime behaviour of a constrained Fokker-Planck equation}

\author{Simon Eberle$^1$}
\address{$^1$Fakultät für Mathematik, Universität Duisburg-Essen.}

\author{Barbara Niethammer$^2$}
\author{André Schlichting$^2$}
\address{$^2$Institut für Angewandte Mathematik, Universität Bonn.}

\email{simon.eberle@uni-due.de}
\email{niethammer@iam.uni-bonn.de}
\email{schlichting@iam.uni-bonn.de}

\let\rho\varrho

\begin{document}



\begin{abstract}
\noindent We consider a Fokker-Planck equation which is coupled to an externally given time-dependent constraint on its first moment.
This constraint introduces a Lagrange-multiplier which  renders the equation nonlocal and nonlinear.

In this paper we exploit an interpretation of this equation as a Wasserstein gradient flow of a free energy ${\mathcal{F}}$  on a time-constrained manifold.
First, we prove existence of solutions by passing to the limit in an explicit Euler scheme obtained by minimizing $h {\mathcal{F}}(\varrho)+W_2^2(\varrho^0,\varrho)$ among all $\varrho$ satisfying the constraint for some $\varrho^0$ and time-step $h>0$.

Second, we provide quantitative estimates for the rate of convergence to equilibrium when the constraint converges to a constant.
The proof is based on the investigation of a suitable relative entropy with respect to  minimizers of the free energy chosen according to the constraint.
The rate of convergence can be explicitly expressed in terms of constants in suitable logarithmic Sobolev inequalities.
\end{abstract}

\maketitle

\section{Introduction}

\noindent We consider a nonlocal Fokker-Planck equation
\begin{equation}\label{e:cFP}
\tau \partial_t \rho(t,x) = \partial_x \left ( \nu^2 \partial_x \rho(t,x) + \left ( H'(x) -\sigma(t) \right ) \rho(t,x) \right ) \ \ \text{with}\ \  \rho(0,x) = \rho^0(x) ,
\end{equation}
which describes the evolution of an ensemble of identical particles in a potential well subject to stochastic fluctuations. Here, a single particle is characterized by its thermodynamic state
\(x\), $H$ is its free energy and
 \(\rho(t,x)\) denotes  the probability density of the whole  system at time \(t\).
Equation \eqref{e:cFP} contains the small parameters~$\tau$ and~$\nu$, where $\nu$ accounts for entropic effects and $\tau$ is the typical relaxation time of a single particle.
Furthermore, $\sigma(t)$ is a Lagrange multiplier which is such that the dynamical constraint
\begin{align}\label{lagrange_multiplier}
\int x\, \rho(t,x) \dx{x} = \ell(t)
\end{align}
is satisfied, where $\ell$ is an externally given constraint.
A direct calculation shows that the Lagrange multiplier is obtained as a nonlocal interaction term
\begin{align}\label{lagranian_multiplier_in_introduction}
\sigma(t) = \int H'(x) \rho(t,x)\dx{x} + \tau\dot\ell(t) .
\end{align}
Equation \eqref{e:cFP} together with \eqref{lagrange_multiplier} was introduced in~\cite{Dreyer2010,DGH11} to model hysteretic behaviour in  many-particle storage systems,
such as for example Lithium-ion batteries subject to externally imposed charging and discharging.
In this context $H$ is typically nonconvex which gives rise to nontrivial dynamics that are studied in various scaling regimes in \cite{HNV12,HNV14}.

System~\eqref{e:cFP} has a free energy, which is also essential in the modeling and its thermodynamic derivation \cite{Dreyer2010,DGH11}, consisting of an entropy and a potential energy
\begin{equation}\label{e:def:freeEnergy}
  \cF(\rho) = \nu^2 \int \rho(x) \log \rho(x) \dx{x} + \int H(x) \; \rho(x) \dx{x} + \log Z_0 .
\end{equation}
Here we added a constant $Z_0 := \int \exp\bra[\big]{-\frac{H(x)}{\nu^2}} \dx{x}$ to make $\cF(\rho)$ nonnegative. By differentiating the free energy along the solution to~\eqref{e:cFP},~\eqref{lagrange_multiplier}, one obtains a second law of thermodynamics for open systems, that is
\begin{equation}\label{e:freeEnergyDissipation}
\pderiv{}{t} \cF(\rho(t)) = - \cD(\rho(t),\sigma(t)) + \tau\, \sigma(t) \, \dot\ell(t) ,
\end{equation}
where the nonnegative dissipation $\cD$ is defined by
\begin{equation}\label{e:def:Dissipation}
 \cD(\rho(t),\sigma(t)) := \int \abs[\big]{ \partial_x \bra*{\nu^2 \log \rho(t,x) + H'(x) - \sigma(t)}}^2  \rho(t,x)\dx{x} .
\end{equation}
If $\dot \ell=0$ the identity \eqref{e:freeEnergyDissipation} is characteristic for systems possessing a gradient flow structure  and  very useful in the investigation of the long-time behaviour of solutions,
since it shows that $\cF$ is a Lyapunov function.
Motivated by this feature, the goal of this paper is two-fold. First we will develop an existence theory for \eqref{e:cFP},\eqref{lagrange_multiplier} based on an underlying
gradient flow formulation. Second, we will provide quantitative estimates of the rate of convergence of solutions to equilibrium in case that $\ell(t) \to \ell^*$ as $t \to \infty$.

We now give a brief overview of the corresponding results.

The first part of this paper on the gradient flow formulation is inspired by the seminal paper \cite{JKO98}, presenting an interpretation of a linear
Fokker-Planck equation as a gradient flow of the free energy~$\cF$ in the space of probability densities with finite second moment endowed with the Wasserstein metric and subject
to the physically accurate free energy functional \eqref{e:def:freeEnergy} (cf.\ also~\cite{Ambrosio2007}). This was the starting point of interpretations of more general nonlinear, nonlocal
Fokker-Planck equations as Wasserstein gradient flows with respect to time-dependent energies, on constrained manifolds and even of none dissipative equations. A selection of works introducing certain constraints into gradient flows are \cite{FerreiraGuevara15,CarlenGagbo03,TudorascuWunsch11}.

In our case, the equation has a time independent energy functional $\cF$, however subject to a possibly time-dependent constraint~\eqref{lagrange_multiplier}. Let us point out here,
that such a setting raises the problem, that the constrained gradient cannot be simply defined as the projection onto the constrained manifold, since this could lead to a violation of the dynamical constraint. In Section~\ref{S:CGF}, it is shown that a gradient flow with respect to a dynamical constraint needs a further restriction of the space of admissible tangential directions in order to match the dynamical constraint at all times. Then, among these admissible tangential directions, the one of steepest descent of the free energy is chosen.

This formal definition is complemented by proving rigorously in Section~\ref{S:time_discrete} that~\eqref{e:cFP} can be obtained as Wasserstein gradient flow with dynamical constraint.
To that aim we introduce time-discrete solutions obtained from an implicit time-discrete Euler scheme. In the setting of geometric flows,
the scheme was first introduced in~\cite{Luckhaus1990,Luckhaus1995} and in the setting of Fokker-Planck equations it goes back to~\cite{JKO98}. Then, equation~\eqref{e:cFP} is obtained by passing
to the limit in the time step of the discrete scheme. In addition, this provides an alternative well-posedness result to \cite[Lemma 1]{HNV14} that is
based on a fixed point argument. Let us also note, that well-posedness in the case of compact state space is also obtained in \cite{Dreyer2011}.

In the second part of this paper, Section~\ref{s:LT}, we prove a quantitative long-time result under the assumption that $\ell(t) \to \ell^*\in \R$ for $t \to \infty$.
In order to identify the internal time-scale of the system we set from now on $\tau=1$.
The main difficulty in the investigation of the long-time behaviour is  that due to the external constraint~\eqref{lagrange_multiplier} the system is not thermodynamically closed,
that is the free energy~\eqref{e:def:freeEnergy} is not strictly decreasing but satisfied the energy-dissipation identity~\eqref{e:freeEnergyDissipation}.
The key idea in the analysis is to introduce
a suitable comparison state $\gamma_{\lambda(\ell)}$ parametrized by the constraint $\ell$ with the help of some convenient reparametrization $\lambda:\R \to \R$. This state is characterized by the constrained minimization of the free energy~\eqref{e:def:freeEnergy} among all states satisfying~\eqref{lagrange_multiplier} (cf.\ Proposition~\ref{prop:constraint:energy:min}).
Then, we are able to establish for the relative entropy with respect to this state, defined by
\begin{equation}\label{e:def:RelEnt}
  \cH\bra[\big]{\rho | \gamma_{\lambda(\ell)}}  := \int  \rho(x) \log \frac{\rho(x)}{\gamma_{\lambda(\ell)}(x)} \dx{x},
\end{equation}
a differential inequality which implies a
quantitative convergence to the equilibrium state. The constant in this differential estimate is characterized by the constant in a suitable logarithmic Sobolev inequality.

Since the relative entropy dominates the $L^1$-norm, we are able to show,  provided $\ell \to \ell^*$ sufficiently fast, that there exists $C< \infty$ depending on the initial value and the convergence assumption on $\ell$ as well as a $c>0$ depending on the constant in the logarithmic Sobolev inequalities such that (cf.\ Theorem~\ref{thm:LT:quant} and Corollaries~\ref{cor:LT:RelEnt*} and~\ref{cor:LT:sigma*})
\begin{equation}\label{e:intro:conv}
  \abs{\sigma(t) - \lambda(\ell^*)} + \int \abs{ \rho(t) - \gamma_{\lambda(\ell^*)}} \dx{x} \leq C e^{-c\, t} .
\end{equation}
Here, we can identify three possible internal time scales: In the \emph{strictly convex} case,  that is $H''(x) \geq k > 0$ for all $x\in \R$, then $c$ in~\eqref{e:intro:conv}
can be chosen as $k/2$. In the \emph{unimodal} case where $H$ has only one global minimum, then $c= \tilde c /\nu^2$ for some $\tilde c>0$ that is independent of $\nu$.
In the so-called \emph{Kramers} case, where  $H$ has a multi-well structure, we obtain that  $c= \tilde c\nu^2 \exp\bra*{ -\Delta H /\nu^2}$ is exponentially small in $\nu$.
Here, $\Delta H$ is a characteristic energy barrier of the system (cf.\ Section~\ref{s:LT:nu}).
Moreover we show that for $\ell^*$ outside of a certain regime and for sufficiently well-prepared initial data, the multi-well structure does not play a role in the dynamics and $c= \tilde c /\nu^2$ for some $\tilde c>0$.

\section{Constrained gradient flows}\label{S:CGF}
\subsection{Setting}\label{S:CGF:Setting}
Let $\cM$ be the state manifold and $\cF: \cM \to \R$ a smooth free energy function. Furthermore, $\cM$ shall possess in each point $u\in \cM$ a tangent space $T_u\cM$ and on~$T_u\cM$ a positive definite symmetric bilinear form $g_u(\cdot,\cdot): T_u \cM \times T_u \cM \to \R$. Then $u: \R_+ \to \cM$ is the gradient flow with respect to $\cF$ if it solves
\begin{equation}\label{eq:defGF}
 g(\partial_t u, s) = - D \cF(u)\cdot s  \qquad \text{ for all } s\in T_u M ,
\end{equation}
where $D\cF(u)\cdot s$ denotes the first variation of $\cF$ at $u$ in direction $s$.

Another formulation (cf.~Mielke~\cite{Mielke2011}) uses the inverse of the metric $g_u$ denoted by the \emph{Onsager operator} $K_u:T^\ast_u \cM \to T_u\cM$, which is assumed to be a positive semidefinite linear operator. Then, the gradient flow of $u$ with respect to $\cF$ is defined by
\begin{equation}\label{eq:defGFOnsager}
  \partial_t u = - K_u D\cF(u) .
\end{equation}
By the definition of the Onsager operator, the cotangent space is given as the preimage of the Onsager operator
\begin{equation}
 \cT_u^*\cM := \set{v : \text{there exists } s\in T_u\cM \text{ such that } K_u v = s } .
\end{equation}
Then, any covector field $v : \R_+ \to \cT_u^* \cM$ gives rise to a curve $\bra[\big]{u(t)}_{t\geq 0}$ on~$\cM$ by solving in a suitable sense
\begin{equation}\label{CGF:CE}
  \partial_t u(t) = K_{u(t)} v(t) .
\end{equation}
We call this the continuity equation on~$\cM$, since it respects possible conservation laws. For instance for $\cM = \cP_2(\R)$ the space of absolutely continuous probability measures with bounded second moment, we formally have $\cT_u \cM= \set*{ s : \int s =0}$ and $K_u v = - \partial_x \bra*{ u \partial_x v}$. Hence, \eqref{CGF:CE} becomes the classical continuity equation on~$\R$: $\partial_t u + \partial_x\bra*{ u \partial_x v} = 0$ with driving potential field $\partial_x v$.

In the following, we often write the identities~\eqref{eq:defGF} or~\eqref{eq:defGFOnsager} as $\partial_t u = - \nabla \cF(u)$ and do not make the underlying metric respectively Onsager operator apparent in the notation. Moreover, we let $\abs{\nabla \cF(u)}_u^2 := g_u(\nabla\cF(u), \nabla\cF(u) )= D\cF(u) \cdot K_u D\cF(u)$.

A crucial consequence of the gradient flow formulation is the so called \emph{energy-dissipation} estimate
\begin{equation}\label{eq:GFeed}
 \pderiv{\cF(u)}{t} =  D\cF(u)\cdot\partial_t u = - D\cF(u) \cdot K_u D\cF(u)=  - \abs{\nabla \cF(u)}^2_u \leq 0 ,
\end{equation}
which corresponds to the second law of thermodynamics for closed systems. In this context the term $\abs{\nabla \cF(u)}^2_u$ is called dissipation.

\subsection{Formalism} \label{S:CGF:Formalism}

In this section we want to introduce our notion of a gradient flow subject to a time dependent constraint. The solution does no longer live on the manifold $\cM$, but for each time $t \geq 0$ the gradient flow has to be an element of a constrained manifold $\cM^\cC$.

Therefore, let $\cC$ be an a-priori given differentiable functional $\cC : \cM \times \R_+ \rightarrow \R$ such that
\begin{align}
D\cC(u,t) \cdot K_u D \cC(u,t) = |\nabla \cC(u,t)|^2_u >0 . \label{e:ass:const:nondeg}
\end{align}
We call a constraint with such a property a nondegenerate constraint and we set $\cM^{\cC}_t := \set{ u \in \cM : \cC(u,t)=0 }$ the time-dependent constrained state space. The constraint is called \emph{stationary} if $\cC(u,t) = \cC(u,0)$ for all $t\in \R_+$.

To define a constrained gradient flow for a stationary constraint,
note first of all that due to the nondegeneracy of $\cC$ from~\eqref{e:ass:const:nondeg}, the gradient $\nabla \cC(u)$ is orthogonal to $T_{u} \cM^\cC$ and $T_{u} \cM^\cC$ is a linear subspace of $T_{u} \cM$  with co-dimension~1.
Let $\pi_\cC$ be the orthogonal projection from $T_{u} \cM$ onto $T_{u} \cM^\cC$. Since $\pi_\cC$ is an orthogonal projection, it is also self-adjoined.
The constrained gradient $\nabla_{\cM^\cC} \cF(u)$ is then the unique element in $T_u\cM^\cC$ such that for all $v \in T_u \cM^\cC$:
\begin{align*}
  g_u(\nabla_\cM \cF(u),v) = g_u(\nabla_{\cM^\cC}  \cF(u),v)
\end{align*}
Employing that $\pi_\cC v = v$ and that $\pi_\cC$ is self-adjoined, we find
\begin{align*}
  g_u(\nabla_{\cM}  \cF(u),v) = g_u(\nabla_{\cM}  \cF(u), \pi_\cC v) = g_u(\pi_\cC \nabla_{\cM}  \cF(u),v)
\end{align*}
and hence, since $v \in T_u \cM^\cC$ was arbitrary,
\begin{align*}
  \nabla_{\cM^\cC}  \cF(u) = \pi_\cC \nabla_\cM\cF(u).
\end{align*}
Therefore, a curve $u: \R_+ \rightarrow \cM$ is called constrained gradient flow with respect to the stationary, nondegenerate constraint $\cC: \cM \rightarrow \R$, if for all $t\geq 0$:
\begin{align*}
  \partial_t u = - \nabla_{\cM^\cC} \cF (u(t)) .
\end{align*}
The additional difficulty in the case of a dynamical constraint is that the orthogonal projection of $\nabla_{\ctM} \ctF (u,t) $ onto $T_{(u,t)} \ctM^\cC$ does not necessarily keep the flow congruent to the dynamical constraint. To keep them synchronized, we introduce an extended state space incorporating the time as an additional coordinate. This approach resembles the basic transformation of a nonautonomous ordinary differential equation to an autonomous one by adding the time as additional coordinate. Hence, let us define the \emph{extended} state manifold $\ctM := \cM \times \R$ and for two elements $(s_1,c_1),(s_1,c_1)\in \ctM$ the formal metric
\begin{align*}
g_{(u,t)}^{\ctM}\bra[\big]{(s_1,c_1), (s_2,c_2)} := g_u^\cM(s_1, s_2) + (c_1 , c_2)_\R .
\end{align*}
Then $\ctM^\cC$ is given by $\ctM \cap \cC^{-1}(\{0\}) = ((\cC(\cdot,t))^{-1}(\{0\}),t)_{t \in \R}$ and the tangent space of $\ctM^\cC$ is given by
\begin{align*}
T_{(u,t)}\ctM^\cC = T_u\cM \cap \bra[\big]{\nabla_\cM \cC(u,t) \times \partial_t \cC(u,t)}^\perp
\end{align*}
In order to lift $\cF$ onto $\ctM$ we define
\begin{align*}
\ctF : \cM \times \R \rightarrow \R, \quad\text{with}\quad  (u,t) \mapsto \cF(u)+t
\end{align*}
and hence
\begin{align*}
\nabla_{\ctM} \ctF (u,t) = (\nabla_\cM \cF(u),1) .
\end{align*}
Therefore, a constrained gradient should satisfy:
\begin{enumerate}[1)]
	\item $\pi_t \nabla_{\ctM^\cC} \ctF =1$
	\item $\nabla_{\ctM^\cC} \ctF(u,t) \in T_{(u,t)} \ctM^\cC \Leftrightarrow \nabla_{\ctM^\cC} \ctF(,t) \perp (\nabla_\cM \cC(u,t), \partial_t \cC(u,t))$
	\item For all $v \in T_u \cM^{\cC(\cdot,t)}$ that is $v \in T_u\cM$ such that $v \perp \nabla_\cM \cC(\cdot,t)$ holds
	\begin{equation}
      g^{\ctM}_{(u,t)}\bra*{\nabla_{\ctM^\cC} \ctF(u,t),(v,0)} = g^\cM_u\bra*{\nabla_\cM \cF(u),v}
	\end{equation}
\end{enumerate}
  Under these premises and the nondegenerate assumption~\eqref{e:ass:const:nondeg} on the constraint, the only possible definition of the constrained gradient flow is the projection of $\nabla_{\ctM} \ctF$ along $(\nabla_\cM \cC(u,t),0 )$ onto $(\nabla_\cM \cC(u,t), \partial_t \cC(u,t))^\perp$.
  Doing so we get
  \begin{align*}
    \nabla_{\ctM^\cC} \ctF(u,t) &= \nabla_{\ctM} \ctF(u,t ) - \sigma(u,t) (\nabla_\cM \cC(u,t),0)  \\
    &= \bra[\Big]{\nabla_\cM \cF{(u)}-\sigma(u,t) \nabla_\cM \cC(u,t),1}
  \end{align*}
  where
  \begin{align}\label{e:CGF:sigma}
    \sigma(u,t):= \frac{g^\cM_u(\nabla_\cM \cF(u),\nabla_\cM \cC(u,t)+\partial_t \cC(u,t)}{|\nabla_\cM \cC(u,t)|_u^2} .
  \end{align}
Hence, we arrive at the following definition for the gradient flow in the case of a dynamical constraint:
\begin{defn}[Dynamically constrained gradient flow] \label{def_dynamic_constraint_grad_flow}
	A curve $u: \R_+ \rightarrow \cM$ is called constrained gradient flow with respect to the nondegenerate dynamical constraint $\cC: \cM \times \R_+ \rightarrow \R$, if for all $t\geq 0$:
	\begin{align}\label{e:def:DynCGF}
	\partial_t u = - \nabla_\cM \cF(u(t))  + \sigma(u(t),t) \; \nabla_\cM \cC(u(t),t)) ,
	\end{align}
	where the Lagrange multiplier $\sigma$ is given by~\eqref{e:CGF:sigma}.
\end{defn}

\subsection{Formal derivation as constrained gradient flow in \texorpdfstring{$(\mathcal{P}_2(\R),W_2)$}{Wasserstein space}}\label{s:GF:cFP:formal}
In this section, we formally show, that \eqref{e:cFP} can be seen as gradient flow with respect to the free energy functional $\cF$ as defined in~\eqref{e:def:freeEnergy} satisfying the constraint~\eqref{lagrange_multiplier}. In the following discussion the parameter $\nu$ is set to one.
The constraint given in terms of a functional $\mathcal{C}: \mathcal{P}_2(\R)\times \R_+ \rightarrow \R$ such that $\cC(\rho,t)=0$ reads
\begin{align}
 \cC(\rho,t) \mapsto M_1(\rho) -\ell(t) \qquad\text{with}\qquad M_1(\rho ) =\int x \rho(x) \dx{x} .
\end{align}
The metric is induced by the Wasserstein distance defined on the space of absolutely continuous probability measure with finite second moment $\cP_2(\R)$ defined by
\begin{equation}
  W_2^2(\rho_0,\rho_1) = \inf_{\pi\in \Pi(\rho_0,\rho_1)} \set*{ \iint \abs{x-y}^2 \; \pi(\dx{x},\dx{y}) },
\end{equation}
where $\Pi(\rho_0,\rho_1)$ is the set of couplings between $\rho_0$ and $\rho_1$, i.e. probability measures on $\R \times \R$ with marginals $\rho_0$ and $\rho_1$, respectively.
Furthermore it holds the dynamical representation~\cite{BenamouBrenier}:
\begin{align*}
  W_2^2(\rho_0,\rho_1) = \inf\biggl\{ \int_0^1 \int \abs{\partial_x v(t,x)}^2 \rho(t,\dx{x}) \dx{t} : &\partial_t \rho(t) = K_{\rho(t)} v(t) , \\ &\rho(0) = \rho_0 , \rho(1)=\rho_1 \biggr\},
\end{align*}
where the Onsager operator $K_\rho$ is defined by $K_\rho v = -\partial_x\bra*{ \rho \, \partial_x v}$ in the weak sense.
The differential of the free energy is given by $D\cF \cdot s = \int \bra*{\log \rho(t) + 1 + H} s \dx{x}$. Moreover, we evaluate
\begin{equation}\label{e:CGF:nablaC}
  \abs{\nabla \cC(\rho,t)}^2_{\rho} = \int \partial_x x \,  \partial_x x\,  \rho(x) \dx{x} = \int \rho(x) \dx{x} = 1 ,
\end{equation}
which satisfies the nondegeneracy assumption~\eqref{e:ass:const:nondeg}. Hence, $\sigma(t)$ in~\eqref{e:CGF:sigma} becomes
\begin{equation}
  \sigma(t) = \int x \partial_x \bra*{ \rho(t) \partial_x \bra*{\log(\rho(t,x)) + 1 + H(x)}} - \dot\ell(t) = \int H'(x) \rho(x) \dx{x} - \dot\ell(t),
\end{equation}
which is $\sigma(t)$ as defined in~\eqref{lagranian_multiplier_in_introduction}. Hence, we obtain from Definition \ref{def_dynamic_constraint_grad_flow}
\begin{align}
  \partial_t \rho(t) &=  \partial_x \bra[\big]{ \rho(t) \partial_x \bra*{ \log \rho(t) + 1 + H} - \partial_x\bra*{\Lambda(t) x}} \\
  &= \partial_{x} \bra*{ \partial_x \rho(t) + \bra*{ H' - \sigma(t)}  \rho(t) } ,
\end{align}
which is nothing else than~\eqref{e:cFP}.

\section{The time discrete scheme and existence of weak solutions}\label{S:time_discrete}
In this section, we make the discussion of Section~\ref{s:GF:cFP:formal} rigorous by showing existence of weak solutions of \eqref{e:cFP} with constraint \eqref{lagrange_multiplier} by using a variational implicit Euler scheme based on the
constrained gradient flow. For the existence, the parameter $\nu$ is set to one.

First, let us fix the assumptions throughout this section and define weak solutions for the constrained Fokker-Planck equations.
\begin{assume}\label{assume:TD}
The function $H\in C^3(\R;\R_+)$ has at most quadratic growth at infinity such that for some $C<\infty$ and all $x\in \R$
\begin{equation}\label{assume:TD:H:growth}
\begin{split}
  &H(x) \leq C (1+ \abs{x}^2) ,\qquad \abs{H'(x)}\leq C (1+\abs{x}), \\
  &\abs{H''(x)} \leq C \qquad\text{and}\qquad \abs{H'''(x)} \leq \frac{C}{1+\abs{x}} .
\end{split}
\end{equation}
The dynamical constraint is Lipschitz, i.e.~$\ell \in W^{1,\infty}(\R_+; \R)$, and the initial data satisfies $\rho^0 \in \cP_2(\R)$ and $\cF(\rho^0)< \infty$.
\end{assume}
\begin{defn}[weak solutions]\label{definition_of_weak_solutions}\label{def:weak_sol}
We say that \(\rho\) is a weak solution of \eqref{e:cFP} and \eqref{lagrange_multiplier} on \( [0,T) \times \R \) for $T>0$, if \(\rho(0) = \rho^0 \) and for a.e. \( t \in (0,T)\) holds \( \rho(t) \in \cM^{\ell(t)} \) with
\begin{equation}\label{def:constraint:mfd}
\cM^{\ell(t)} := \set[\bigg]{\rho \in \mathcal{P}_2(\R) \ \Big|\ \ell(t) = \int x \, \rho(x) \dx{x} }
\end{equation}
as well as \(\rho\) is a distributional solution of \eqref{e:cFP}, that is for all $\zeta \in C_c^\infty([0,T) \times \R)$ holds
\begin{align}
\begin{split}
0 = &- \int_\R  \zeta(x,0) \rho^0(x)\dx{x}
 - \int_0^T  \int_\R \rho(t,x) \partial_t \zeta(t,x)\dx{x}\dx{t} \\
&+ \int_0^T \int_{\R} \rho(t,x) \left ( \left (H'(x) -\sigma(t) \right )\partial_x \zeta(t,x) -\partial_{xx} \zeta(t,x) \right )\dx{x}\dx{t}.
\end{split}
\end{align}
\end{defn}
For the proof of existence of a weak solution of \eqref{e:cFP} with constraint \eqref{lagrange_multiplier} we follow mainly the ideas in \cite{JKO98} which are to use a Wasserstein gradient flow with respect to the free energy functional \(\mathcal{F}\)~\eqref{e:def:freeEnergy}. This gradient flow is carried out in a time discrete manner for arbitrary but fixed time-step length $h$ and leading thereby to a sequence of piecewise constant approximations $\rho_h(t,x)$ of the solution. Finally the limit $h \rightarrow 0$ is taken and it is proven that the limit $\rho(t,x)$ actually is a weak solution of \eqref{e:cFP} with constraint \eqref{lagrange_multiplier}.
The main additional difficulty in comparison to \cite{JKO98}, is the need for additional estimates on the Lagrange multiplier and the second moment.

\subsection{The Euler scheme}

Since the metric and Onsager operator $K_u$ are induced by the Wasserstein distance (cf.\ also \cite{JKO98,AGS05}), we use the following time-discrete variational approximation. Let $h>0$ be a fixed time step and consider the following constrained implicit Euler scheme
\begin{equation}\label{discrete_scheme}
\rho^k := \argmin\limits_{\rho \in \cM^{\ell(kh)}}\bra*{ \tfrac{1}{2} W_2^2( \rho^{k-1} , \rho) +h \mathcal{F}(\rho) } .
\end{equation}
In the following, we often investigate the entropy and potential energy inside of the free energy~\eqref{e:def:freeEnergy} separately and write $\cF(\rho) = S(\rho) + E(\rho) + \log Z_0$ with
\begin{equation}
  S(\rho) := \int \rho(x) \log \rho(x) \dx{x} \qquad\text{and}\qquad E(\rho) := \int H(x) \rho(x) \dx{x} .
\end{equation}
First of all we show the well-posedness of the scheme.
\begin{prop}[Well-posedness of the scheme]\label{Proposition_well_posedness_of_the_scheme}
Given $\rho^0 \in \cM^{\ell(0)}$, there is a unique sequence $(\rho^k)_{k \in \mathbb{N}}$ satisfying the scheme \eqref{discrete_scheme}.
\end{prop}
The proof mainly follows the respective proof in \cite[Proposition 4.1]{JKO98} which is using the direct method to show existence and exploiting the strict convexity of the functional \(\mathcal{F}\) and the convexity of $\cM^{\ell(kh)}$ for the uniqueness. The only additional step is to show, that the first moment is preserved along the minimizing sequence in the direct method is preserved (cf.~\cite[Proposition 4.2]{thesis}).

\subsection{Passing to the Limit}
In this section we show that a constant in time interpolation of the solution of the discrete scheme \eqref{discrete_scheme}  leads to a weak solution of~\eqref{e:cFP} with constraint \eqref{lagrange_multiplier}. Again we follow the structure of the respective proof outlined in \cite[Theorem 5.1]{JKO98}. However, the additional constraint~\eqref{lagrange_multiplier} on the first moment leads to the rise of a Lagrange multiplier~\eqref{lagranian_multiplier_in_introduction} which needs to be extracted from the discrete scheme~\eqref{discrete_scheme}.
\begin{thm}[Existence of a weak solution]\label{existence_of_weak_solutions}
Suppose Assumption~\ref{assume:TD} holds. For fixed $h>0$, let $(\rho_h^k)_{k \in \mathbb{N}}$  be the solution of the scheme \eqref{discrete_scheme}.  Define the constant interpolation $\rho_h: (0, \infty) \times \R \rightarrow [0, \infty)$ by
\begin{equation}\label{e:def:const:interpolation}
\rho_h(t,x) = \rho_h^k(x) \quad \text{ for } t \in [kh, (k+1)h) \text{ and } k \in \mathbb{N} .
\end{equation}
Then for any $T>0$ and $h \to 0$
\begin{equation}\label{weak_L1_convergence_of_the_constant_interpolation}
\rho_h  \rightharpoonup  \rho \quad \text{ weakly in } L^1((0,T)\times \R)
\end{equation}
and $\rho(t, \cdot) \in \cM^{\ell(t)}$ is a weak solution of~\eqref{e:cFP} with constraint~\eqref{lagrange_multiplier}.
Moreover it holds
\begin{equation}\label{uniform_bound_on_second_moment_and_energy}
M_2(\rho) , E(\rho) \in L^\infty((0,T)) .
\end{equation}
\end{thm}
\begin{rem}[Regularity, energy dissipation and uniqueness]
The regularity of the solutions constructed in Theorem~\ref{existence_of_weak_solutions} can be improved.
The only difference  to the unconstrained case is the Lagrange multiplier $\sigma$. However, the uniform bounds for \(M_2(\rho)\) and \(\sigma\) already contained
in \cite[Appendix A Proposition 2]{HNV14} (cf.\ also Lemma~\ref{lem:bound:M2sigma})
ensure that we are able to apply standard regularity results for the Fokker-Planck equation (cf.~\cite{JKO98} and \cite[Chapter 5]{thesis}) to obtain
\begin{equation}\label{regularity}
   \rho \in L^2_{\loc}(\R_+; H^2_{\loc}(\R)) \cap H^1_{\loc}(\R_+ ; L^2_{\loc}(\R)) .
\end{equation}
The regularity can be further improved under stronger assumptions on the potential $H$ and external constraint $\ell$ (cf.~\cite[Theorem 5.2]{thesis} for a detailed statement).

Using the improved regularity properties~\eqref{regularity} a chain rule is established, which rigorously shows the energy-dissipation identity~\eqref{e:freeEnergyDissipation}. Similarly, the improved regularity is sufficient to establish uniqueness by a comparison argument. Due to the nonlocal nature of the equation, the strategy for the uniqueness proof
in~\cite{JKO98} has to be modified. Instead of proving uniqueness for the solutions itself, by following the idea of \cite{HNV14} one considers the distribution function, which allows for a comparison principle (cf.~\cite[Chapter 6]{thesis}).
\end{rem}
The proof of Theorem~\ref{existence_of_weak_solutions} is based on the following three Lemmas, which are proved separately in the next section. The first one provides an approximate weak formulation of the Fokker-Planck equation~\eqref{e:cFP} with constraint~\eqref{lagrange_multiplier}.
\begin{lem}[Time-discrete approximation of the weak formulation]\label{lem:time-discrete-weak}
The solution $(\rho_h^k)_{k\in \N}$ to the discrete scheme~\eqref{discrete_scheme} satisfies for all $\zeta\in C_c^\infty(\R)$ and all $k\in \N$
\begin{equation}\label{discrete_weak_formulation}
\bigg | \int_\R \frac{\rho^{k}_h-\rho^{k-1}_h}{h} \zeta + \int \bra[\big]{ (H'(x)- \sigma^k_h ) \partial_x  \zeta - \partial_{xx}  \zeta } \rho^k  \bigg | \leq \sup \limits_\R \frac{| \partial_{xx} \zeta |}{2} \  \frac{1}{h} \, W_2^2(\rho^{k-1}_h, \rho^k_h) ,
\end{equation}
where the discrete Lagrange multiplier $\sigma^k_h$ is given by
\begin{equation}\label{easy_version_of_sigma_k_h}
\sigma^k_h = \int_\R H'(x) \rho^k_h(x) \dx{x} +  \frac{1}{h} \int_\R \left ( \rho^k_h(x) - \rho^{k-1}_h(x) \right ) x \dx{x} .
\end{equation}
\end{lem}
In the next Lemma, we establish a priori bounds, which allow to pass to the limit.
\begin{lem}[A priori estimates for the discrete scheme]\label{lem:time-discrete-apriori}
Let $(\rho_h^k)_{k\in \N}$ be the solution of the scheme~\eqref{discrete_scheme}.
Then for any $T>0$ there exists $C>0$ such that for all $h \in (0,1)$ and all $N \in \mathbb{N}$ with $Nh \leq T$ the  following a priori estimates hold true
\begin{align}
\sup_{0\leq k\leq N} M_2(\rho_h^k) &\leq C \label{uniform_bound_on_second_moment}\\
\sup_{0\leq k\leq N} \int_\R \max \{\rho_h^N \log\rho_h^N,0 \} \dx{x} &\leq C \label{uniform_bound_on_entropy_positive_part} \\
\sup_{0\leq k\leq N} E(\rho_h^k) &\leq C \label{uniform_bound_on_energy}\\
\sum \limits_{k=1}^N W_2^2(\rho_h^{k-1},\rho_h^k) &\leq Ch \label{uniform_bound_on_distance_sum} .
\end{align}
\end{lem}
Based on the a priori estimates of Lemma~\ref{lem:time-discrete-apriori}, the only additional difficulty in the passage to the limit in the approximate weak formulation~\eqref{discrete_weak_formulation} is the convergence of the Lagrange multiplier $\sigma^k_h$. Hence, we prove its uniform convergence separately in the next Lemma.
\begin{lem}[Convergence of the Lagrange multiplier]\label{lem:time-discrete-sigma}
  Let $(\rho_h^k)_{k\in \N}$ be the solution of the scheme~\eqref{discrete_scheme} and define
  \begin{equation}\label{e:def:sigma_h}
    \sigma_h(t) := \sum \limits_{k=0}^\infty \sigma_h^k \ \mathds{1}_{[kh,(k+1)h)}(t).
  \end{equation}
  with $\sigma_h^k$ given by~\eqref{easy_version_of_sigma_k_h}. Then, there exists $\sigma \in C([0,T];\R)$ such that
  \begin{equation}\label{e:sigma_h:to:sigma}
    \sigma_h \to \sigma  \qquad\text{uniformly on $(0,T)$} .
  \end{equation}
\end{lem}
With the help of the above three Lemmas, we can prove Theorem~\ref{existence_of_weak_solutions}.
\begin{proof}[Proof of Theorem~\ref{existence_of_weak_solutions}]
The a priori estimates of Lemma~\ref{lem:time-discrete-apriori} allow us to pass in the piecewise constant interpolation $\rho_h$ from~\eqref{e:def:const:interpolation} to the limit $h\rightarrow 0$.
Indeed, from \eqref{uniform_bound_on_second_moment} we derive tightness of $(\rho_h)_{h > 0}$ and from \eqref{uniform_bound_on_entropy_positive_part}, since $[0, \infty) \ni x \mapsto \max\{x \log(x),0\}$ has superlinear growth, we deduce that for any $T>0$ holds up to a subsequence
\begin{equation}\label{weak_L1_convergence_in_0-T_times_R}
\rho_h \rightharpoonup \rho \quad \text{ in } L^1((0,T) \times \R) .
\end{equation}
In addition the tightness implies that $\rho(t,\cdot)$ is a probability density for a.e.~$t\in (0,T)$, which shows \eqref{weak_L1_convergence_of_the_constant_interpolation}. To prove $\rho(t, \cdot) \in \cM^{\ell(t)}$ for a.e. $t \in (0,T)$, we note that the a priori estimate~\eqref{uniform_bound_on_second_moment} passes to the limit and we have $M_2(\rho(t,\cdot)) \leq C$ for all $t\in (0,T)$.
Similarly, for showing that $\rho(t,\cdot)\in \cM^{\ell(t)}$, we use from the construction the identity
\begin{align}
\frac{1}{2 \delta} \int_{t-\delta}^{t+\delta} \sum \limits_{kh\leq T} \ell(kh) \, \mathds{1}_{[kh,(k+1)h)}(\tau)\dx{\tau}
 =  &\frac{1}{2 \delta} \int_{t-\delta}^{t+\delta} \int_\R x  \, \rho_h(\tau,x)\dx{x}\dx{\tau}.
\end{align}
The second moment bound~\eqref{uniform_bound_on_second_moment} implies enough tightness to pass to the limit in the identity. By the growth assumption~\eqref{assume:TD:H:growth} on $H$, the statement $E(\rho(t,\cdot)) \in L^\infty((0,T))$ follows along the same argument as the proof for $M_2(\rho(t,\cdot)) \in L^\infty((0,T))$, which completes the proof of \eqref{uniform_bound_on_second_moment_and_energy}.

\noindent It remains to show that $\rho$ solves \eqref{e:cFP}. Therefore, we sum \eqref{discrete_weak_formulation} from $k=1,\dots, N$, use the a priori estimate~\eqref{uniform_bound_on_distance_sum} and obtain for any $\zeta \in C_c^\infty([0,T) \times \R)$ by using the definition of $\sigma_h$ from~\eqref{e:def:sigma_h} the estimate
\begin{equation}\label{almost_final_sorted_discrete_version}
\begin{split}
\bigg \vert &\frac {1}{h} \int_{T-h}^T \int_\R \zeta(t,x) \rho_h(t,x) \dx{x}\dx{t} -  \frac {1}{h} \int_0^h \int_\R \zeta(t+h,x) \rho_h(t,x) \dx{x}\dx{t}
   \\
&- \int_{(h,T-h)} \int_\R \frac{\zeta(t+h,x)- \zeta(t,x)}{h} \rho_h(t,x) \dx{x}\dx{t} \\
&+ \int_{(h,T)} \int_\R \bra[\big]{ \left (  H'(x) - \sigma_h(t) \right ) \partial_x \zeta(t,x) -\partial_{xx} \zeta(t,x) } \rho_h(t,x) \dx{x}\dx{t} \bigg  \vert \\
&\leq \sup \limits_\R \frac{\big | \partial_{xx} \zeta(t,x ) \big |}{2} \ \sum \limits_{k=1}^{N-1} W_2^2(\rho_h^{k-1}, \rho_h^k) .
\end{split}
\end{equation}
To arrive at the weak formulation of \eqref{e:cFP}, it is left to pass to the limit on both sides. The right-hand side goes to zero for $h \rightarrow 0$, which follows directly from \eqref{uniform_bound_on_distance_sum}. For passing to the limit on the left hand side, we use~\eqref{e:sigma_h:to:sigma} from Lemma~\ref{lem:time-discrete-sigma} and finally obtain
\begin{align}
\begin{split}
0 = &- \int_\R  \zeta(x,0) \rho^0(x)\dx{x}
 - \int_{(0,T)} \int_\R \rho(t,x) \partial_t \zeta(t,x)\dx{x}\dx{t} \\
&+ \int_{(0,T)} \int_\R \rho(t,x) \left ( \left (H'(x) -\sigma(t) \right )\partial_x \zeta(t,x) -\partial_{xx} \zeta(t,x) \right )\dx{x}\dx{t} ,
\end{split}
\end{align}
which by Definition~\ref{definition_of_weak_solutions} is a weak solution to the constrained Fokker-Planck equation.
\end{proof}

\subsection{Proof of auxiliary Lemmas~\ref{lem:time-discrete-weak}--\ref{lem:time-discrete-sigma}}
\begin{proof}[Proof of Lemma~\ref{lem:time-discrete-weak}]
We can choose $h$ arbitrary but fixed, hence we neglect it in the notation of the proof.
Since $\rho^k$ minimizes \eqref{discrete_scheme} among all admissible probability densities $\rho \in \cM^{\ell(kh)}$, the Euler-Lagrange equation has to ensure that perturbations of $\rho^k$ are still in $\cM^{\ell(kh)}$. The perturbations are realized as a push-forward with respect to the flow of a smooth vector field.

In contrast to \cite{JKO98}, we have to use a second push-forward as correction to ensure the constraint on the first moment is met. This second push-forward causes the Lagrange multiplier. Note, that the idea is the same as in Section~\ref{S:CGF:Formalism} and we would like to choose the constant vector field $x\mapsto 1$ corresponding to $\nabla \cC$ (cf.\ \eqref{e:CGF:nablaC}). However, this vector field is not in $C_c^\infty(\R)$ and we work with an admissible one and use an approximation argument at the end of the proof.

Let $\xi \in C_c^\infty(\R)$ and let $\Phi_\alpha$ be the flow with respect to $\xi$, that is the solution of
\begin{equation}
\partial_\alpha \Phi_\alpha = \xi \circ \Phi_\alpha \quad \text{ for all } \alpha \in \R \quad \text{ and } \Phi_0 = \Id .
\end{equation}
Then the push-forward of $\rho^k$ with respect to $\Phi_\alpha$ denoted by $\rho^k_\alpha := \Phi_{\alpha \#} \rho^k$ is given by
\begin{equation}
  \int_\R \varphi(x) \rho^k_\alpha(x) \dx{x} = \int_\R  \varphi \circ \Phi_\alpha(x) \rho^k(x) \dx{x} \quad \text{ for all } \varphi \in C_b^0(\R) .
\end{equation}
For the correction take another vector field $\eta \in C_c^\infty(\R)$, that satisfies the nondegeneracy property $\int_\R \eta(x) \, \rho^k(x) \dx{x} \neq 0$ ensuring that the respective push-forward is able to change the first moment of $\rho^k$. We define the flow $\Psi_\beta$ with respect to~$\eta$:
\begin{equation}
\partial_\beta \Psi_\beta = \eta \circ \Psi_\beta \quad \text{ for all } \beta \in \R \quad \text{ and } \Psi_0 = \Id .
\end{equation}
Let the joint push-forward be given as
\begin{equation}
\rho^k_{\alpha, \beta} := \Psi_{\beta \#} (\Phi_{\alpha \#} \rho^k ) .
\end{equation}
To make sure that $\rho_{\alpha,\beta}^k \in \mathcal{P}_2(\R)$, it needs to be shown that $M_2(\rho_{\alpha,\beta}^k) < \infty$.
First of all observe, that $\Psi_\beta \circ \Phi_\alpha(x) = x$ on $\R \setminus ( \supp \xi \cup \supp \eta )$ and $\supp \xi \cup \supp \eta  \subset [-K,K]$ for some $K>0$. Hence, we can always approximate by using a spatial cut-off function to arrive at the estimate
\begin{align} \label{perturbation_finite_second_moment}
M_2(\rho_{\alpha,\beta}^k) = \int_\R \bra*{\Psi_\beta \circ \Phi_\alpha(x)}^2 \rho^k(x) \dx{x} \leq \sup \limits_{x \in [-K,K]}  \bra*{\Psi_\beta \circ \Phi_\alpha(x)}^2+ M_2(\rho^k)
\end{align}
Another approximation argument with a spatial cut-off function, ensures the identity
\begin{align} \label{first_moment_of_perturbed_density_representation}
M_1(\rho^k_{\alpha,\beta}) = \int_\R (\Psi_\beta \circ \Phi_\alpha)(x) \, \rho^k(x) \dx{x}
\end{align}
The smoothness of the flows allows to differentiate the above functional
\begin{align}
\frac{\partial M_1(\rho^k_{\alpha,\beta})}{\partial \alpha} \bigg \vert_{(0,0)} &\stackrel{\eqref{first_moment_of_perturbed_density_representation}}{=}  \lim \limits_{\alpha \rightarrow 0} \frac{1}{\alpha}  \int_\R (\Phi_\alpha - \Id) \rho^k(x) \dx{x} = \int_\R \xi(x) \rho^k(x) \dx{x} \\
\frac{\partial M_1(\rho^k_{\alpha,\beta})}{\partial \beta} \bigg \vert_{(0,0)} &\stackrel{\eqref{first_moment_of_perturbed_density_representation}}{=}  \lim \limits_{\beta \rightarrow 0} \frac{1}{\beta}  \int_\R (\Psi_\beta - \Id) \rho^k(x) \dx{x} = \int_\R \eta(x) \rho^k(x) \dx{x} \neq 0
\end{align}
Now, the constraint $\rho_{\alpha,\beta}^k\in \cM^{\ell(kh)}$ reads $M_1(\rho^k_{\alpha,\beta}) \stackrel{!}{=} \ell(kh)$.  Hence, by the implicit-function-theorem, due to the nondegenerate property of $\eta$, there is some $\eps>0$ and a function  $i \in C^1\bra*{ (-\eps,\eps) ; \R}$  such that $M_1\bra[\big]{\rho^k_{\alpha, i(\alpha)}} = \ell(kh)$. To identify the Lagrange multiplier, we note that $i'(0)$ is given by
\begin{equation}
i'(0) = - \frac{\partial_\alpha M_1\bra[\big]{\rho^k_{\alpha, \beta}}\big|_{(0,0)}}{\partial_\beta M_1\bra[\big]{\rho^k_{\alpha, \beta}}\big|_{(0,0)}} =  - \frac{\int_\R \xi(x) \rho^k\dx{x}}{\int_\R \eta(x) \rho^k\dx{x}}\label{differential_of_implicit_function} .
\end{equation}
For the Euler-Lagrange-equation, we proceed to calculate
\begin{align*}
&\text{\emph{a)} } \pderiv{}{\alpha} \bigg\vert_0 E(\rho^k_{\alpha, i(\alpha)}) \qquad\qquad \text{\emph{b)} } \pderiv{}{\alpha} \bigg \vert_0 S(\rho^k_{\alpha, i(\alpha)}) \\
&\text{\emph{c)} } \limsup \limits_{\alpha \rightarrow 0} \frac{1}{\alpha} \left ( \tfrac{1}{2} W_2^2(\rho^{k-1}, \rho^k_{\alpha, i(\alpha)}) - \tfrac{1}{2} W_2^2(\rho^{k-1}, \rho^k) \right ) .
\end{align*}
\emph{Ad a)} By using monotone convergence, we can use $H$ as a test function in the push-forward and have
$\int_\R H(x) \rho^k_{\alpha, i(\alpha)}(x) \dx{x} = \int_\R  H(\Psi_{i(\alpha)} \circ \Phi_\alpha(x)) \rho^k(x) \dx{x}$.
Therefore, it holds
\begin{align*}
\frac{d}{d \alpha} \bigg \vert_0 E(\rho^k_{\alpha, i(\alpha)})
 &= \lim \limits_{\alpha \rightarrow 0} \frac{1}{\alpha} \int_\R \left ( H(\Psi_{i(\alpha)} \circ \Phi_\alpha(x)) - H(x) \right ) \rho^k(x) \dx{x} \\
&= \int_\R H'(x) \left ( \xi(x)+i'(0) \eta(x) \right )  \rho^k(x) \dx{x} .
\end{align*}
Here, we used again that $\xi$ and $\eta$ have compact support and thus $\tfrac{1}{\alpha}(H(\Psi_{i(\alpha)} \circ \Phi_\alpha(x)) - H(x))$ converges uniformly to $\pderiv{}{\alpha} \big\vert_0 H(\Psi_{i(\alpha)} \circ \Phi_\alpha)$.

\noindent\emph{Ad b)} The use of $\log(\rho^k_{\alpha, i(\alpha)})$ as a test function in the push forward is justified by a both sided truncation and an application of Lusin's Theorem to the truncated function.
Moreover, for $\alpha$ small enough, the function $\alpha \mapsto \Psi_{i(\alpha)} \circ \Phi_\alpha$ is strictly monotone and we can apply the transformation rule to find the following identity
\begin{align}
 \rho^k_{\alpha , i(\alpha)}\bra*{\Psi_{i(\alpha)} \circ \Phi_\alpha(x)} = \frac{\rho^k(x)}{\partial_x\bra*{ \Psi_{i(\alpha)} \circ \Phi_\alpha(x)}}  .
\end{align}
Thus, we obtain
\begin{align}
\int_\R \log (\rho^k_{\alpha , i(\alpha)}(x)) \rho^k_{\alpha , i(\alpha)}(x) \dx{x} &= \int_\R \log( \rho^k_{\alpha , i(\alpha)}(\Psi_{i(\alpha)} \circ \Phi_\alpha(x))) \rho^k(x) \dx{x} \\
&= \int_\R \log \left ( \frac{\rho^k(x)}{\partial_x (\Psi_{i(\alpha)} \circ \Phi_\alpha(x))} \right )\rho^k(x) \dx{x} .
\end{align}
Therefore it holds:
\begin{align*}
\frac{d}{d \alpha} \bigg \vert_0 S(\rho^k_{\alpha, i(\alpha)}) &= - \lim \limits_{\alpha \rightarrow 0} \int_\R  \frac{1}{\alpha} \log \left (\partial_x (\Psi_{i(\alpha)} \circ \Phi_\alpha(x)) \right ) \rho^k(x) \dx{x} \\
&= -\int_\R \left ( \partial_x \xi +i'(0) \partial_x \eta \right ) \rho^k(x) \dx{x} .
\end{align*}

\noindent\emph{Ad c)} Let $p$ be an optimal coupling of $\rho^{k-1}$ and $\rho^k$, then $p_\alpha$ satisfying for all $\varphi \in C^0_b(\R \times \R)$
\begin{equation}
\int_{\R \times \R} \varphi(x,y) \, p_\alpha(\dx{x},\dx{y}) = \int_{\R \times \R} \varphi(x, \Psi_{i(\alpha)} \circ \Phi_\alpha(y)) \, p(\dx{x},\dx{y})
\end{equation}
is a coupling of $\rho^{k-1}$ and $\rho^k_{\alpha, i(\alpha)}$.
Therefore it holds by the usual truncation and approximation arguments
\begin{align*}
W_2^2(\rho^{k-1}, \rho^k_{\alpha, i(\alpha)}) \leq  \int_{\R \times \R}  | \Psi_{i(\alpha)} \circ \Phi_\alpha(y) -x|^2 \, p(\dx{x},\dx{y})
\end{align*}
which implies
\begin{align*}
\limsup \limits_{\alpha \rightarrow 0} \ &\frac{1}{\alpha} \left ( \tfrac{1}{2} W_2^2(\rho^{k-1}, \rho^k_{\alpha, i(\alpha)}) - \tfrac{1}{2} W_2^2(\rho^{k-1}, \rho^k) \right )\\
&\leq  \int_{\R \times \R} (y-x) \left ( \xi(y)+i'(0) \eta(y) \right ) p(\dx{x},\dx{y}) .
\end{align*}
Since \(\rho^k_{\alpha,i(\alpha)}\) is a perturbation of the minimizer \(\rho^k\) it holds
\begin{align}
\frac{d}{d\alpha} \bigg |_0 \left ( \frac{1}{2} W_2^2(\rho^{k-1}, \rho^k_{\alpha,i(\alpha)}) + h \mathcal{F}(\rho^k_{\alpha,i(\alpha)}) \right ) = 0 .
\end{align}
Together with \emph{a), b), c)} applied also to the reversed vector fields $\xi \rightarrow -\xi$ and $\eta \rightarrow - \eta$ as well as noting that this does not change the sign of $i'(0)$), we obtain
\begin{align}\label{discrete_euler_lagrange_intermediate_equation}
\begin{split}
0 = &\int_{\R \times \R} (y-x) \left ( \xi(y)+i'(0) \eta(y) \right ) p(\dx{x},\dx{y}) \\
&+ h \int_\R \left ( H'(x) ( \xi(x) + i'(0) \eta(x) ) -\partial_x ( \xi(x) + i'(0) \eta(x) ) \right ) \rho^k(x) \dx{x} .
\end{split}
\end{align}
To arrive at a discrete approximation of the weak formulation of \eqref{e:cFP} with constraint~\eqref{lagrange_multiplier}, it is necessary to find an explicit representation of $p$ and generate the time difference. Therefore, the following estimate for any $\varphi \in C_c^\infty(\R)$ is used
\begin{equation}
\begin{split}
&\bigg|\int_\R (\rho^k - \rho^{k-1} ) \varphi(y)\dx{y} -  \int_{\R \times \R}  (y-x) \varphi'(y) \, p(\dx{x},\dx{y})\bigg |   \\
&=   \bigg | \int_{\R \times \R} \bra[\big]{\varphi(y)-\varphi(x) - (y-x) \varphi'(y) } \, p(\dx{x},\dx{y})\bigg | \\
&\leq \frac{1}{2} \sup \limits_\R |\varphi''| \int_{\R \times \R} (y-x)^2 \, p(\dx{x},\dx{y}) = \frac{1}{2} \sup \limits_\R |\varphi''| ~ W_2^2(\rho^{k-1}, \rho^k)
\end{split}
\end{equation}
Let $\chi, \zeta \in C_c^\infty(\R)$ be such that $\chi' = \eta$ and $\zeta' = \xi$. Now, we use $\varphi = \zeta+ i'(0)\chi$ in the above estimate together with \eqref{discrete_euler_lagrange_intermediate_equation} to obtain
\begin{equation*}
\begin{split}
&\bigg | \int_\R \frac{\rho^{k}-\rho^{k-1}}{h} \left ( \zeta + i'(0) \chi \right) + \bra[\big]{ H'(x) \partial_x ( ( \zeta + i'(0) \chi ) - \partial_{xx} ( \zeta + i'(0) \chi ) } \rho^k(x) \dx{x} \bigg | \\
&\leq \frac{1}{2}~ \sup \limits_\R | \partial_{xx} ( \zeta + i'(0) \chi ) |~~ \frac{1}{h} ~ W_2^2(\rho^{k-1}, \rho^k) .
\end{split}
\end{equation*}
In combination with \eqref{differential_of_implicit_function}, which now reads $i'(0)= - \int_\R \partial_x \zeta ~\rho^k\dx{x}  / \int_\R \partial_x \chi ~\rho^k\dx{x}$, we get
\begin{align}\label{very_first_version_of_sth_like_discrete_weak_formulation}
\begin{split}
&\bigg | \int_\R \frac{1}{h} \left ( \rho^{k}-\rho^{k-1} \right ) \zeta + \left ( H'(x) \partial_x  \zeta - \partial_{xx}  \zeta \right ) \rho^k\dx{x} - \sigma^k \int_\R \partial_x \zeta ~\rho^k\dx{x} \bigg | \\
&\leq \frac{1}{2}~ \sup \limits_\R | \partial_{xx} ( \zeta + i'(0) \chi ) |~~ \frac{1}{h} ~ W_2^2(\rho^{k-1}, \rho^k) .
\end{split}
\end{align}
with
\begin{equation}\label{e:def:sigmak:p}
  \sigma^k :=  \frac{1}{h} \frac{\int_\R (\rho^k -\rho^{k-1}) \chi\dx{x}}{\int_\R \partial_x \chi ~\rho^k\dx{x}} + \frac{\int_\R H' \partial_x \chi ~\rho^k\dx{x}}{\int_\R \partial_x \chi ~\rho^k\dx{x}} + \frac{\int_\R \partial_{xx} \chi ~\rho^k\dx{x}}{\int_\R \partial_x \chi ~\rho^k\dx{x}} .
\end{equation}
To simplify the above expression, we use the correction vector field~$\chi(x) = x$, which obviously satisfies the nondegenerate assumption, since \(\int_\R \partial_x \chi(x) \rho^k(x)\dx{x} = \int_\R \rho^k(x)\dx{x} =1\). To approximate $\chi$, let \(\varphi\) be such that
\begin{align}
\varphi \in C_c^\infty(\R,[0,1]), \quad \varphi = 1 \text{ on } [-1,1] , \quad \varphi = 0 \text{ on } \R \setminus [-2,2] , \quad |\varphi'|, |\varphi''| \leq C
\end{align}
and set $\varphi_M(x):= \varphi \left (\frac{x}{M} \right)$ and hence $\chi_M(x):= x \varphi_M(x)$. Then, it holds $\chi_M \in C_c^\infty(\R)$, $\chi_M(x) = x$  on $[-M,M]$, $\chi_M(x) = 0$ on $[-2M,2M]$, $|\chi'_M| \leq C$ and $|\chi''_M| \leq C$ for all $M >0$. Since $\rho^k$ and $\rho^{k-1}$ have finite second moments, we can use $|\rho^k - \rho^{k-1}| |x|$ as an integrable majorant for $(\rho^k - \rho^{k-1}) \chi_M$.
Finiteness of the second moment of $\rho^k$ and the at most linear growth of $H'$ at infinity give that $C |H'| \rho^k$ is a integrable majorant for $H' \partial_x \chi_M \rho^k$.
Due to $\rho^k$ being probability densities and therefore in~$L^1(\R)$, $C\rho^k$ is a simple integrable majorant for $\partial_{xx} \chi_M \rho^k$ and $\partial_{x} \chi_M \rho^k$.
Therefore taking $\chi =\chi_M$ in \eqref{e:def:sigmak:p} and applying Lebesgue's dominated convergence theorem (for $M \rightarrow \infty$) we arrive at~\eqref{easy_version_of_sigma_k_h}. Finally, comparing with~\eqref{very_first_version_of_sth_like_discrete_weak_formulation}, we have also proven~\eqref{discrete_weak_formulation}.
\end{proof}

\medskip
\begin{proof}[Proof of Lemma~\ref{lem:time-discrete-apriori}]
The $\rho_h^k$ are constructed by the scheme \eqref{discrete_scheme} and we want to derive the a priori estimates from this relation. Especially, we would like to use $\rho^{k-1}_h$ at time step $k$ as a test function. However, since it may hold $\ell(kh) \neq \ell((k-1)h)$ and hence $\rho_h^{k-1} \not\in \cM^{\ell(kh)}$ we have to make a perturbation. Let us define the translation of $\rho_h^{k-1}$ as follows:
\begin{align}
_{a_{k-1}}\rho_h^{k-1} := \rho^{k-1}_h(\cdot -{a_{k-1}})
\end{align}
where \({a_{k-1}}\) is chosen such that \(M_1(_{a_{k-1}}\rho_h^{k-1})= M_1(\rho_h^{k-1})+{a_{k-1}} \overset{!}{=} M_1(\rho^k_h)\). Now, $_{a_{k-1}}\rho_h^{k-1}$ is admissible in \eqref{discrete_scheme}, which leads to the basic bound for all $k$ s.t.~$ kh \leq T$
\begin{equation}\label{sheme_estimate}
\tfrac{1}{2}W_2^2(\rho^k_h,\rho^{k-1}_h) +h\mathcal{F}(\rho^k_h) \leq \tfrac{1}{2} W_2^2(_{a_{k-1}}\rho_h^{k-1}, \rho_h^{k-1}) +h\mathcal{F}(_{a_{k-1}}\rho_h^{k-1}).
\end{equation}
Summing this estimate from $k=1, \dots, N$, we arrive at
\begin{align}
\frac{1}{2} \sum \limits_{k=1}^N W_2^2(\rho^k_h,\rho^{k-1}_h) &\leq \frac{1}{2} \sum \limits_{k=1}^N W_2^2(_{a_{k-1}}\rho_h^{k-1}, \rho_h^{k-1}) +h  \sum \limits_{k=1}^N \left (\mathcal{F}(_{a_{k-1}}\rho_h^{k-1}) -\mathcal{F}(\rho^k_h) \right) \notag \\
&=\frac{1}{2} \sum \limits_{k=1}^N \underbrace{W_2^2(_{a_{k-1}}\rho_h^{k-1}, \rho_h^{k-1})}_{=:\text{(I$_k$)}} \label{estimate_for_the_distance_sum}\\
&\phantom{=} +h  \bra[\Bigg]{ \sum \limits_{k=2}^{N} \underbrace{ \left (\mathcal{F}(_{a_{k-1}}\rho_h^{k-1}) -\mathcal{F}(\rho^{k-1}_h) \right )}_{=:\text{(II$_k$)}}  +\mathcal{F}(\rho^0) - \mathcal{F}(\rho^N_h) } \notag.
\end{align}
The term (I$_k$) is easily estimated by using the shift $x\mapsto x+a_{k-1}$ as transport map
\begin{equation*}
W_2^2(_{a_{k-1}}\rho_h^{k-1}, \rho^{k-1}) \leq \int_{\R} |x-(x+ a_{k-1})|^2 \rho^{k-1}_h(x) \dx{x}  =a_{k-1}^2 .
\end{equation*}
Now, we use the growth assumption~\eqref{assume:TD:H:growth} on $H$ to estimate  (II$_k$)
\begin{equation}\label{perturbation_estimate_on_free_energy}
\begin{split}
\MoveEqLeft{\mathcal{F}(_{a_{k-1}}\rho_h^{k-1}) -\mathcal{F}(\rho^{k-1}_h) = \int_\R \left ( H(x+{a_{k-1}}) -H(x) \right ) \rho_h^{k-1}(x) \dx{x} }\\
&\leq \int_\R \left (|H'(x) {a_{k-1}}| +\tfrac{1}{2} a_{k-1}^2 \sup \limits_{y \in [x,x+{a_{k-1}}]} |H''(y)| \right) \rho_h^{k-1}(x) \dx{x} \\
&\leq C |{a_{k-1}}| \left ( \int_\R x^2 \rho_h^{k-1}(x) \dx{x} + 1 +  |{a_{k-1}}| \right ) \\
&\leq C |{a_{k-1}}| \left ( M_2(\rho_h^{k-1}) + 1 + \abs{a_{k-1}} \right ) .
\end{split}
\end{equation}
Since $\ell$ is assumed to be Lipschitz, there exists $L>0$ such that $|{a_{k-1}}| = |\ell(kh)-\ell((k-1)h)| \leq Lh$. Moreover, we use the standard lower bound on the entropy \cite[Equation (14)]{JKO98} in terms of the second moment, that is $S(\rho) \geq -C\bra*{ M_2(\rho) +1}^{\alpha}$ for any $\alpha\in \bra*{\tfrac{1}{3},1}$. Since $H\geq 0$ by Assumption~\ref{assume:TD}, we can conclude that
\begin{equation}\label{energy_lower_estimate}
  -\mathcal{F}(\rho_h^N) \leq C \bra{ M_2(\rho_h^N) + 1}^{\alpha} \leq C \bra{ M_2(\rho_h^N) + 1} .
\end{equation}
Applying these bounds to \eqref{estimate_for_the_distance_sum} yields
\begin{align}
\frac{1}{2} \sum \limits_{k=1}^N W_2^2(\rho^k_h,\rho^{k-1}_h)
\leq C h \bra*{ T + \sum \limits_{k=0}^{N}  M_2(\rho_h^{k})} \label{estimate_for_the_distance_sum_refined}
\end{align}
for some uniform constant $C$ only depending on $L$ and Assumption~\ref{assume:TD}.
Hence, once we have established \eqref{uniform_bound_on_second_moment}, also~\eqref{uniform_bound_on_distance_sum} follows. To estimate the difference-quotient of the second moment, we use \eqref{discrete_weak_formulation} with $\zeta = x^2$, which is justified by the finiteness of the second moment of each of the $\rho^k$ and the growth assumption~\eqref{assume:TD:H:growth} on $H'$
\begin{align*}
\MoveEqLeft{\frac{1}{h} \left ( M_2(\rho_h^k) -M_2(\rho^{k-1}_h) \right ) = \frac{1}{h} \int_\R x^2 \left ( \rho^k_h - \rho^{k-1}_h \right )\dx{x}} \\
&\stackrel{\mathclap{\eqref{discrete_weak_formulation}}}{\leq} \ \bigg | \int_\R \left ( \left ( H'- \sigma_h^k \right ) 2x -2 \right ) \rho_h^k(x) \dx{x} \bigg | + \frac{1}{h} W_2^2 \left (\rho^k_h,\rho^{k-1}_h \right ) \\
&\leq \ \underbrace{ 2\int_\R |H'| |x| \rho^k_h(x) \dx{x}}_{=:\text{(I)}} + \underbrace{2|\sigma_h^k| \int_\R |x| \rho^k_h(x) \dx{x}}_{=:\text{(II)}} +2 + \frac{1}{h} {W_2^2 \left (\rho^k_h,\rho^{k-1}_h \right )}  .
\end{align*}
By the growth assumption~\eqref{assume:TD:H:growth} follows $\text{(I)}\leq C \bra*{ 1+ M_2(\rho_h^k)}$. By using the definition of $\sigma^k_h$ from~\eqref{easy_version_of_sigma_k_h} follows
\begin{align*}
\text{(II)} &\leq 2 \left ( C \int_\R (|x|+1) \rho_h^k\dx{x} + L \right ) \int_\R |x| \rho_h^k(x) \dx{x} \leq C \bra*{1+  M_2(\rho_h^k)} .
\end{align*}
Altogether, we arrive at
\begin{align}
\frac{1}{h} \left ( M_2(\rho_h^k) -M_2(\rho_h^{k-1}) \right ) \leq C\bra*{1 +  M_2(\rho_h^k)} +  \frac{1}{h} {W_2^2 \left (\rho^k_h,\rho^{k-1}_h \right )} .
\end{align}
Summing from $k=1, \dots,N$ and multiplying by $h$ yields
\begin{align*}
M_2(\rho_h^N) - M_2(\rho^0) &\ \leq\ C hN + C h \sum \limits_{k=1}^N M_2(\rho_h^k) + \sum \limits_{k=1}^N W_2^2(\rho_h^k, \rho_h^{k-1})  \\
&\ \stackrel{\mathclap{\eqref{estimate_for_the_distance_sum_refined}}}{\leq}\  C T + C h \sum \limits_{k=1}^N M_2(\rho_h^k) .
\end{align*}
Hence, a discrete Gronwall argument gives the estimate
\begin{align}
M_2(\rho_n^N) \leq (M_2(\rho^0) +C) \exp( CT),
\end{align}
which establishes \eqref{uniform_bound_on_second_moment} and by~\eqref{estimate_for_the_distance_sum_refined} also~\eqref{uniform_bound_on_distance_sum}. By theses bounds, we can rewrite~\eqref{estimate_for_the_distance_sum} as
\begin{align}\label{estimate_for_the_distance_sum_refined_after_bounded_second_moment}
\frac{1}{2} \sum \limits_{k=1}^N W_2^2(\rho_h^k, \rho_h^{k-1} ) \leq CT h + h \left ( \mathcal{F}(\rho^0) - \mathcal{F}(\rho_h^N) \right )  .
\end{align}
Now, the estimates~\eqref{uniform_bound_on_entropy_positive_part} and~\eqref{uniform_bound_on_energy} follow by using once more the lower bound $S(\rho)\geq - \bra*{M_2(\rho)
+1}^\alpha$ for $\alpha\in (\tfrac{1}{3},1)$ along the lines as in \cite{JKO98}
\begin{alignat}{2}
\int_\R \max \{\rho_h^N \log(\rho_h^N),0 \}\dx{x} &\ \leq\ S(\rho_h^N) + \int_\R | \min \{\rho_h^N \log(\rho_h^N),0 \} |\dx{x} \\
&\ \leq\  S(\rho_h^N) +C (M_2(\rho_h^N)+1)^\alpha \\
&\ \leq\ \mathcal{F}(\rho_h^N) +C (M_2(\rho_h^N)+1)^\alpha \\
&\stackrel{\mathclap{\eqref{estimate_for_the_distance_sum_refined_after_bounded_second_moment}}}{\ \leq\ }  \mathcal{F}(\rho^0) +CT+C (M_2(\rho_h^N)+1)^\alpha,
\end{alignat}
which by \eqref{uniform_bound_on_second_moment} yields \eqref{uniform_bound_on_entropy_positive_part}. Similarly, we can estimate
\begin{alignat}{2}
E(\rho_h^N) &= \mathcal{F}(\rho_h^N) - S(\rho_h^N) - \log Z_0 \stackrel{\mathclap{\eqref{estimate_for_the_distance_sum_refined_after_bounded_second_moment}}}{\ \leq\ } \mathcal{F}(\rho^0) +C T+C (M_2(\rho_h^N)+1)^\alpha + C,
\end{alignat}
which again with \eqref{uniform_bound_on_second_moment} gives \eqref{uniform_bound_on_energy}.
\end{proof}

\medskip
\begin{proof}[Proof of Lemma~\ref{lem:time-discrete-sigma}]
The proof is based on an Arzelà-Ascoli argument. Since $\sigma_h$ is not continuous, we approximate it by a continuous function as follows
\begin{equation*}
\tilde{\sigma}_h(t) := \sigma_h^k + \left (\sigma_h^{k+1}-\sigma_h^k \right ) \frac{t-kh}{h} \qquad \text{for }t \in [kh, (k+1) h)
\end{equation*}
To apply the Arzelà-Ascoli argument, it is necessary to proof uniform boundedness and a uniform modulus of continuity of $(\tilde{\sigma}_h(t))_{h >0}$.

\noindent\emph{Uniform boundedness.} By using the Lipschitz bound on $\ell$ as well as the growth estimates on $H$ from Assumption~\ref{assume:TD}, it holds
\begin{equation}\label{boundedness_of_sigma}
\begin{split}
\sup  \limits_{h>0, t \in (0,T)}|\tilde{\sigma}_h(t)| &\leq  \max_{\{k,h : kh \leq T\}} |\sigma_h^k| \\
&\stackrel{\mathclap{\eqref{easy_version_of_sigma_k_h}}}{=} \max\limits_{\{k,h| kh \leq T\}} \bigg \vert \frac{1}{h} \int_\R \left (\rho_h^k- \rho^{k-1}_h \right ) x\dx{x} + \int_\R H'(x) \rho_h^k(x) \dx{x} \bigg \vert \\
&\leq\max \limits_{\{k,h| kh \leq T\}} \left (  L + C(1+M_2(\rho^k_h))\right ) \stackrel{\eqref{uniform_bound_on_second_moment}}{\leq} C .
\end{split}
\end{equation}
\noindent\emph{Uniform modulus of continuity.} We prove an Lipschitz bound, by first noting that for any $t,t'\in (0,T)$, it holds
\begin{align*}
\frac{| \tilde{\sigma}_h(t) - \tilde{\sigma}_h(t') |}{|t-t'|} \leq
   \max_{k\in \set{1,\dots,N}} \frac{|\sigma_h^{k} -\sigma_h^{k-1}|}{h}
\end{align*}
Hence, it suffices to estimate the increment for any $k$ with $kh \leq T$. Therefore, based on the definition of $\sigma_h^k$ from~\eqref{easy_version_of_sigma_k_h} we apply \eqref{discrete_weak_formulation} with $H'$ instead of $\zeta$, which is justified by an approximation using dominated convergence and the second moment estimate, and arrive at
\begin{align*}
\MoveEqLeft{\frac{|\sigma_h^k - \sigma_h^{k-1} |}{h} \stackrel{\eqref{easy_version_of_sigma_k_h}}{=}\frac{1}{h} \bigg | \frac{1}{h} \int_\R \left ( \rho_h^k - \rho_h^{k-2} \right ) x\dx{x} + \int_\R H'(x)   ( \rho_h^k - \rho_h^{k-1} )\dx{x}  \bigg | }
\\
&\stackrel{\eqref{discrete_weak_formulation}}{\leq} \left ( \bigg | \frac{1}{h^2} \int_\R \left ( \rho_h^k - \rho_h^{k-2} \right ) x\dx{x} \bigg | + \bigg | \int_\R \left ( \left (H'(x) - \sigma^k_h \right )  H''(x) -H'''(x) \right ) \rho_h^k\dx{x} \bigg | \right . \\
&\qquad +  \left .  \sup_{x \in \R} \frac{|H'''(x)|}{2} \ \frac{1}{h}W_2^2(\rho_h^{k-1} , \rho_h^k ) \right ) \\
&\leq  \Bigg(\bigg | \frac{1}{h^2} \left ( \ell(kh) - \ell((k-2)h) \right ) \bigg |
+ \bigg | \int_\R H'(x) H''(x) \rho_h^k(x) \dx{x} \bigg | \\
&\qquad\quad + \bigg | \sigma_h^k\int_\R H''(x) \rho_h^k\dx{x} \bigg | + \tfrac{1}{2} \sup \limits_\R |H'''(x)| \frac{1}{h} W_2^2(\rho_h^{k-1},\rho_h^k) \Bigg) \\
&=: \text{(I)} +\text{(II)} + \text{(III)} + \text{(IV)}
\end{align*}
Now, each term is bounded by using Assumption~\ref{assume:TD} and the a priori estimate of Lemma~\ref{lem:time-discrete-apriori}. Indeed, it holds $\text{(I)} \leq \sup_{t\in [0,T]} \abs{\ddot{l}(t)} \leq C$ since $\ell \in C^2(\R_+;\R)$. Then, the growth condition on $H$ as stated in~\eqref{assume:TD:H:growth} imply $\text{(II)} \leq C \int_R (1+ \abs{x}) \rho_h^k(x) \dx{x} \leq C$ by \eqref{uniform_bound_on_second_moment}. Likewise the already proven uniform bound on $\sigma_h^k$ in~\eqref{boundedness_of_sigma} and uniform assumption on $H''$ imply $\text{(III)} \leq C$. Finally, the term $\text{(IV)}$ is bounded by the uniform assumption on $H'''$ and the a priori estimate~\eqref{uniform_bound_on_distance_sum}. Hence, we have proven the discrete Lipschitz estimate for any $t,t'\in [0,T]$
\begin{equation}
 \frac{| \tilde{\sigma}_h(t) - \tilde{\sigma}_h(t') |}{|t-t'|} \leq
   \max_{k\in \set{1,\dots,N}} \frac{|\sigma_h^k - \sigma_h^{k-1} |}{h} \leq C \label{uniform_modulus_of_continuity}
\end{equation}
Therefore, by an Arzelà-Ascoli argument exists $\sigma\in C([0,T];\R)$ such that $\tilde{\sigma}_h \rightarrow \sigma$ uniformly on $(0,T)$ along a subsequence. It remains to estimate the error done by the linear interpolation of $\sigma_h(t)$
\begin{equation}
\sup \limits_{t \in (0,T)} |\sigma_h(t) -\tilde{\sigma}_h(t)| \leq \sup \limits_{(k+1)h \leq T} |\sigma_h^{k+1}-\sigma_h^k| \stackrel{\eqref{uniform_modulus_of_continuity}}{\leq} Ch \quad \text{ for all } h>0 .
\end{equation}
\end{proof}

\section{Long time behaviour}\label{s:LT}

In this section we investigate the evolution of the constrained Fokker-Planck equation \eqref{e:cFP} under the assumption that the external forcing becomes constant and under quadratic growth assumption at infinity of the potential \(H\). The general idea is based on exploiting the entropy-dissipation identity~\eqref{e:freeEnergyDissipation}. This strategy was partly also applied in \cite[Chapter 5]{Dreyer2011} and \cite[Chapter 7]{thesis} to derive the qualitative trend to equilibrium.
We complemented this result with a quantitative rate of convergence to equilibrium based on the investigation of suitable relative entropies with respect to local equilibrium sates. Therefore, let us first characterize these states and prove some auxiliary results

We set the parameter $\nu$ to one within the next two sections and discuss the $\nu$-dependence of the constants later in Section~\ref{s:LT:nu}.
\subsection{Local equilibrium states and first properties}\label{s:LT:qual}
For the qualitative long-time behaviour, we make the following slightly stronger assumptions in comparison to Assumption~\ref{assume:TD} on the potential $H$ and forcing term $\ell$.
\begin{assume}\label{assume:LT:qual}
  The function $H\in C^3(\R;\R_+)$ has quadratic growth at infinity such that for two constants $c_{H,\pm} >0$
  \begin{equation}\label{assume:LT:H:quadratic_growth}
    \liminf_{x \rightarrow \pm \infty} H''(x) = c_{H,\pm} \qquad\text{ and }\qquad \lim \limits_{x \rightarrow \pm \infty}H'''(x) = 0 .
  \end{equation}
  The forcing $\ell \in W^{1,\infty}(\R_+; \R)$ becomes stationary
  \begin{equation}\label{assume:LT:l:const}
    \lim \limits_{t \rightarrow \infty} \ell(t) = \ell^*  \quad\text{ and } \quad  \quad \dot\ell \in L^1(\R_+) .
  \end{equation}
  The initial data $\rho^0 \in \cP_2(\R)$ has finite free energy~\eqref{e:def:freeEnergy}, that is $\mathcal{F}(\rho^0) < \infty$.
\end{assume}
A simple example for a potential, that satisfies these assumptions is a double-well-potential with quadratic growth at infinity like $H(x) = (\sqrt{x^2+1}-2)^2$.

Let us introduce a family $\set*{\gamma_{\sigma}}_{\sigma\in \R}$ of probability measures on $\R$ parametrized by the Lagrange multiplier $\sigma$ with density defined by
\begin{equation}\label{def:gamma}
  \gamma_{\sigma}(x) := \frac{1}{Z_{\sigma}} \exp\bra[\big]{ -H(x) + \sigma\, x}  \quad\text{with}\quad Z_{\sigma} := \int \exp\bra[\big]{ -H(x) + \sigma\, x} \dx{x} .
\end{equation}
Note, that under the growth assumption~\eqref{assume:LT:H:quadratic_growth} the partition sum $Z_\sigma$ is finite. Moreover, for a probability density $\gamma\in \cP_2(\R)$, we define its variance by
\begin{equation}\label{e:def:var}
  \Var(\gamma) := \int x^2 \gamma(x) \dx{x} - \bra*{ \int x \gamma(x) \dx{x}}^2 .
\end{equation}
First of all we characterize the energy minimizer for a constant constraint.
\begin{prop}[Constrained free energy minimization]\label{prop:constraint:energy:min}
  The constrained minimization of the free energy functional~$\mathcal{F}$ as given in~\eqref{e:def:freeEnergy} over the constrained manifold $\cM^{\ell}$ defined in~\eqref{def:constraint:mfd} has a unique minimum given by $\gamma_{\lambda(\ell)}$ as in~\eqref{def:gamma}. Here, the function $\lambda : \R \to \R$ is implicitly defined as the solution to
  \begin{equation}\label{def:lambda}
    \forall \ell \in \R : \qquad M_1\bra*{\gamma_{\lambda(\ell)}} \stackrel{!}{=} \ell  .
  \end{equation}
  In addition, by defining the constants
  \begin{equation}\label{def:cvar}
    c_{\Var} := \inf_{\eta\in \R} \Var(\gamma_{\eta}) \qquad \text{and}\qquad C_{\Var}  := \sup_{\eta\in \R} \Var(\gamma_{\eta})
  \end{equation}
  it holds the bi-Lipschitz estimate
  \begin{equation}\label{e:lambda:deriv}
    0 < c_{\Var} \leq \pderiv{M_1\bra*{\gamma_{\lambda}}}{\lambda} = \Var(\gamma_{\lambda}) \leq C_{\Var} < \infty .
  \end{equation}
\end{prop}
\begin{proof}
The uniqueness of the minimizer follows from strict convexity and weak-\(L^1(\R)\)-lower semicontinuity of $\cF$ over the weakly $L^1(\R)$-closed and convex set $\cM^\ell$ by the direct method of the calculus of variations (cf.~Proposition \ref{Proposition_well_posedness_of_the_scheme} and \cite[Proposition 4.1]{JKO98}).

To characterize the minimizer $\rho^\ell$, we use the convexity of $x \mapsto x \log(x)$ on \(\R_+\) and obtain
\begin{align}
x \log(x) \geq y \log(y) + (\log(y)+1)(x-y) \quad \text{ for } x \geq 0 , y >0 .
\end{align}
Then, for all $\rho \in \cM^\ell$, we find
\begin{align*}
\mathcal{F}(\rho) &=  \int \bra*{ \rho \log \rho   + H \rho } \dx{x}
 \geq \int \bra*{ \rho^\ell \log \rho^\ell + \bra[\big]{ \log(\rho^\ell) +1 } (\rho-\rho^\ell) + H \rho}  \dx{x} \\
 &= \mathcal{F}(\rho^\ell) + \int \left ( \log(\rho^\ell) +H(x) - \lambda(\ell)x - z(\ell) \right ) (\rho - \rho^\ell) \dx{x}
\end{align*}
Here, by using the fact that $\int \bra[\big]{\lambda(\ell) x+z(\ell) } (\rho -\rho^\ell) \dx{x} =0$ for $\rho , \rho^\ell \in \cM^\ell$, we introduced two additional Lagrange multipliers $z(\ell), \lambda(\ell)\in \R$ corresponding to the conversation of total mass and the constraint $M_1(\rho)=\ell$, respectively. Since the lower bound has to hold for all $\rho\in \cM^\ell$, we obtain that $\log(\rho^\ell) +H(x) - \lambda(\ell)x - z(\ell)=0$ and by uniqueness in $\cM^\ell$, $\rho^\ell$ is exactly of the form $\rho^{\ell} = \gamma_{\lambda(\ell)}$ as defined in~\eqref{def:gamma} and $\ell\mapsto \lambda(\ell)$ yet to be determined. Therefore, it remains to show that for any $\ell \in \R$ there is a unique $\lambda(\ell)$, such that $M_1(\gamma_{\lambda(\ell)}) = \ell$. We calculate the derivative of~$\lambda\mapsto M_1(\gamma_{\lambda})$
\begin{align*}
\pderiv{M_1(\gamma_{\lambda})}{\lambda} &= \pderiv{}{\lambda} \int  x \, \gamma_{\lambda}(x) \dx{x} = \pderiv{}{\lambda} \int  x \, \frac{\exp\bra*{-H(x)+\lambda\, x}}{Z_{\lambda}} \dx{x} \\
&= \int x^2 \, \gamma_\lambda(x) \dx{x} - \bra*{ \int x \, \gamma_\lambda(x) \dx{x}}^2 = \Var(\gamma_\lambda) .
\end{align*}
Assumption~\ref{assume:LT:qual} ensures that the constants $c_{\Var}$ and $C_{\Var}$ defined in~\eqref{def:cvar} are positive and finite, respectively.
Indeed, for $\lambda \to \pm \infty$, the quadratic growth~\eqref{assume:LT:H:quadratic_growth} ensures that~$\gamma_{\lambda}$ converges to a Gaussian distribution with standard deviation $c_{\pm}^{-1}$, respectively, implying a finite and positive variance. Therefore, the function $\lambda \to M_1(\gamma_{\lambda})$ is uniformly bi-Lipschitz and especially strictly monotone with $M_1(\gamma_{\lambda}) \to \pm\infty$ as~$\lambda\to \pm \infty$.
\end{proof}
It is convenient to consider the relative entropy as given in~\eqref{e:def:RelEnt} with respect to the measures $\gamma_{\sigma}$ as defined in~\eqref{def:gamma}.

By comparing the above definition with the free energy~\eqref{e:def:freeEnergy}, we observe the identity $\cH(\rho | \gamma_0) = \cF(\rho)$. We need to compare the relative entropies with the dissipation. The according energy--dissipation inequality is called a logarithmic Sobolev inequality and is well studied in the literature (cf.~\cite{Gross1975,Ledoux1999,Arnold2001,MS12}).
\begin{lem}[Logarithmic Sobolev inequality]\label{lem:LSI}
  Under Assumption~\ref{assume:LT:qual}, the measure $\gamma_{\sigma}$ from~\eqref{def:gamma} satisfies for any $\sigma \in \R$ a logarithmic Sobolev inequality with constant $C_{\LSI}(\sigma)>0$, that is for all $\rho\in \cP_2(\R)$ with $\cD(\rho,\sigma)<\infty$ it holds
  \begin{equation}\label{def:LSI}
    \cH(\rho | \gamma_{\sigma}) \leq C_{\LSI}(\sigma) \; \cD(\rho,\sigma)
  \end{equation}
  Moreover, for any $M>0$ exists $C_{\LSI,M}< \infty$ such that
  \begin{equation}\label{e:LSI:constant}
    \sup_{\abs{\sigma}\leq M} C_{\LSI}(\sigma) \leq C_{\LSI,M} .
  \end{equation}
  Moreover, if $H$ is uniformly convex, that is if for some $k>0$ it holds $\inf_{x\in\R} H''(x) \geq k$, then $C_{\LSI} = \frac{1}{k}$ independent of $\sigma$.
\end{lem}
\begin{proof}
  Using the function $H_\sigma(x) := H(x) - \sigma x$, we can write $\gamma_{\sigma} = e^{-H_\sigma}/Z_\sigma$ with $Z_\sigma$ as defined in~\eqref{def:gamma}. Then
  by Assumption~\eqref{assume:LT:H:quadratic_growth} on the growth of $H$ it follows, that there exists a decomposition of $H_{\sigma} = H_{c,\sigma} + H_{b,\sigma}$ into two functions $H_{c,\sigma}, H_{b,\sigma} : \R \to \R$ such that $H_{c,\sigma}$ is uniformly convex and $H_{b,\sigma}$ is compactly supported and bounded such that
  \begin{equation*}
    \inf_{x\in \R} H''_{c,\sigma}(x) \geq \frac{\min\set{c_{H,+},c_{H,-}}}{2} \quad\text{and} \quad  \sup_{x\in \R} H_{b,\sigma} - \inf_{x\in \R} H_{b,\sigma} \leq C_{H,\sigma} .
  \end{equation*}
  The measure $\gamma_{\sigma}$ is of the form such that \cite[Corollary 1.7]{Ledoux1999} can be applied and we conclude
  \begin{equation*}
    C_{\LSI}(\sigma) \leq \frac{2 e^{2 C_{H,\sigma}}}{\min\set{c_{H,-},c_{H,+}}}.
  \end{equation*}
  Statement~\eqref{e:LSI:constant} is an immediate consequence, since $C_{H,\sigma}$ and in particular $C_{\LSI}(\sigma)$ depends by the smoothness of $H$ continuously on $\sigma$ and hence is bounded on each compact interval.
  Finally, in the convex case, we can directly apply \cite[Corollary 1.6]{Ledoux1999}.
\end{proof}
To apply the logarithmic Sobolev inequality, we have to ensure that $\sigma$ is bounded, which is the content of the following lemma.
\begin{lem}\label{lem:bound:M2sigma}
Under Assumption \eqref{assume:LT:qual}, it holds for any solution to \eqref{e:cFP} with constraint~\eqref{lagrange_multiplier}
\begin{enumerate}[ (i) ]
 \item The second moment and the Lagrange multiplier $\sigma$ remain bounded, that is for some $C>0$ we have
\begin{align}\label{e:bound:M2sigma}
  \|M_2(\rho(\cdot))\|_{L^\infty(\R_+)} \leq C  \qquad\text{and}\qquad  \|\sigma\|_{L^\infty(\R_+)} \leq C  .
\end{align}
\item The free energy along solutions remains bounded, that is for some $C>0$ it holds
\begin{align}\label{energy_and_second_moment_bound_long_time}
 \| \mathcal{F}(\rho(\cdot))\|_{L^\infty(\R_+)} \leq C .
\end{align}
\item For any sequence $\set{t_n}_{n\in\N}$ with $t_n \to \infty$ there exists a subsequence $t_{n(k)}$ such that
\begin{align}\label{weak_convergence_for_subsequences_long_time}
\begin{split}
 \rho_{t_{n(k)}} \rightharpoonup \rho^* \quad \text{ in } L^1(\R)
 \qquad\text{ with }\qquad \rho^* \in \cM^{\ell^*} .
 \end{split}
\end{align}
\end{enumerate}
\end{lem}
\begin{proof}
\emph{(i):} This is content of~\cite[Appendix A, Proposition 2]{HNV14}. The basic idea is to use a comparison principle for scalar ODEs and the quadratic growth
Assumption~\eqref{assume:LT:H:quadratic_growth} on $H$, to establish
the bound $M_2\bra*{\rho(t,\cdot)} \leq C\bra[\big]{1+ \norm{\sigma}_{L^\infty(\R_+)}}$. By using again the explicit quadratic growth, the equation~\eqref{lagranian_multiplier_in_introduction}
yields $\abs{\sigma(t)} \leq C\bra[\big]{1+ M_2(\rho(t,\cdot))}$. Hence, together it follows that $\norm{\sigma}_{L^\infty(\R_+)}
\leq C \bra[\big]{1+ \norm{\sigma}_{L^\infty(\R_+)}}^{1/2}$.

\emph{(ii):} The upper bound for the free energy follows from the energy-dissipation identity~\eqref{e:freeEnergyDissipation}
\begin{align}
\mathcal{F}(\rho(t)) \leq \mathcal{F}(\rho^0) + \int_0^\infty \sigma(t) \dot\ell(t) \dx{t} \leq \mathcal{F}(\rho^0) + \|\sigma\|_{L^\infty(\R_+)} \|\dot\ell\|_{L^1(\R_+)} .
\end{align}
The lower bound for the free energy follows by estimating the negative part of the entropy function in terms of the second moment. For any $\alpha \in ( \tfrac{1}{3} , 1)$ exists $C>0$ such that $\int \rho \log \rho \geq -C (M_2(\rho) +1)^\alpha$ (cf.\ \cite[Equation (14)]{JKO98}), which together with $(i)$ gives the lower bound.

\emph{(iii):}
The sequence \((\rho(t_n))_{n \in \mathbb{N}}\) is uniformly integrable. Indeed, we obtain from~\eqref{energy_and_second_moment_bound_long_time} the bound
\begin{align}\label{auxiliary_estimate_bound_on_positive_part_of_entropy_long_time}
\int \bra[\big]{\rho(t_n) \log(\rho(t_n))}_+ \dx{x} \leq C ,
\end{align}
where for $a\in \R$, $\bra{a}_+ := \max\set{a,0}$ denotes the positive part.
This estimate implies uniform integrability of \((\rho(t_n))_{n \in \mathbb{N}}\) via
\begin{align}
\int_{\{|\rho(t_n)| >M\}} \rho(t_n,\dx{x}) &\leq \frac{1}{\log(M)} \int_{\{|\rho(t_n)| >M\}} \bra[\big]{\rho(t_n) \log(\rho(t_n))}_+  \dx{x} \leq \frac{C}{\log(M)} ,
\end{align}
where $M>e$ is arbitrary.
Therefore, we find a subsequence \((\rho(t_{n_k}))_{k \in \mathbb{N}}\) such that $\rho(t_{n_k}) \rightharpoonup \rho^*$  in $L^1(\R)$. It remains to show \(\rho^* \in \cM^{\ell^*}\).

The uniform bound on the second moment in \eqref{e:bound:M2sigma} implies tightness, which implies $\int_\R \rho^* \dx{x} = 1$. Similarly, the convergence of the first moment and boundedness of the second moment follow from a standard truncations argument from the bound~\eqref{e:bound:M2sigma}. This shows $\rho^* \in \cM^{\ell^*}$ and finishes the proof.
\end{proof}
\begin{rem}
  From the results of Lemma~\ref{lem:bound:M2sigma}, it is possible to identify the limit $\rho^*$ as $\gamma_{\lambda(\ell^*)}$ and to show that $\rho(t_n) \rightharpoonup \gamma_{\lambda(\ell^*)}$ as well as $\cF(\rho(t_n)) \to \cF(\gamma_{\lambda(\ell^*)})$ (cf.~\cite[Chapter 7]{thesis}) for some sequence $(t_n)_{n\in\N}$.
\end{rem}

\subsection{Convergence to equilibrium in relative entropy}\label{s:LT:quant}

The previous section indicates, that under Assumption~\ref{assume:LT:qual} the free energy of solutions $\cF(\rho(t))$ converges to $\cF(\gamma_{\sigma^*})$ whenever the constraint becomes constant $\ell(t)\to \ell^*$ with $\sigma^*$ such that $\int x \gamma_{\sigma^*} = \ell^*$.
This suggests, that the relative entropy $\cH(\rho(t)| \gamma_{\sigma^*}) \to 0$ as $t\to \infty$ as defined in~\eqref{e:def:RelEnt}. In order to prove this, one could seek for a differential inequality involving the relative entropy with respect to~$\gamma_{\sigma^*}$. However, a direct approach in this direction needs to show convergence of the Lagrangian multiplier $\sigma$, since terms involving $\sigma - \sigma^*$ would occur along the calculation.

To avoid the occurrence of terms involving $\sigma$, which cannot easily be controlled, we introduce a quasistationary equilibrium following the constraint, which is given by $\gamma_{\lambda(\ell(t))}$ with $\lambda(\ell)$ as defined in Proposition~\ref{prop:constraint:energy:min}. It turns out, that the relative entropy with respect~$\gamma_{\lambda(\ell(t))}$ allows for a control without the need to show convergence of~$\sigma$.

The first observation is the following relative entropy comparison as well as a comparison of certain free energy differences with relative entropy.
\begin{lem}\label{lem:RelEnt:comp}
  For all $\eta \in \R$, all $\ell \in \R$ and all $\rho \in \cM^\ell$ it holds
  \begin{equation}\label{e:RelEnt:lambda}
    \frac{c_{\Var}}{2} \bra*{ \eta - \lambda(\ell)}^2 \leq \cH\bra*{\rho | \gamma_\eta} - \cH\bra*{\rho | \gamma_{\lambda(\ell)}} \leq \frac{C_{\Var}}{2} \bra*{ \eta - \lambda(\ell)}^2
  \end{equation}
  as well as
  \begin{equation}\label{e:FreeEnergy:RelEnt:Ident}
    \cF(\rho) - \cF(\gamma_\eta) - \cH(\rho | \gamma_\eta) = \eta \bra[\big]{ \ell - M_1(\gamma_{\eta})} .
  \end{equation}
  Moreover, for any $\ell^*\in \R$, any $\ell \in \R$ and all $\rho \in \cM^\ell$ it holds
  \begin{equation}\label{e:RelEnt:lambda:cor}
    \cH(\rho | \gamma_{\lambda(\ell^*)}) \leq \cH(\rho | \gamma_{\lambda(\ell)}) + \frac{C_{\Var}}{2 c_{\Var}^2} \abs*{ \ell^* - \ell}^2 .
  \end{equation}
\end{lem}
\begin{proof}
 Let us first rewrite for any $\eta,\lambda \in \R$ and all $\rho \in \cM^\ell$ the relative entropy difference
  \begin{align*}
    \cH\bra*{\rho | \gamma_{\eta}} -  \cH\bra{ \rho | \gamma_{\lambda}} &= \log \frac{Z_\eta}{Z_\lambda} + \ell\bra*{ \lambda - \eta}
    = \int_\lambda^\eta \bra*{ \partial_\xi \log Z_{\xi} - \ell } \dx{\xi} \\
    &= \int_{\lambda}^\eta \bra*{M_1(\gamma_\xi) - \ell} \dx{\xi} .
  \end{align*}
  Choosing $\lambda = \lambda(\ell)$ as defined in~\eqref{def:lambda}, it follows $\ell = M_1(\gamma_{\lambda(\ell)})$ and hence
  \begin{align*}
    \cH\bra*{\rho | \gamma_{\eta}} -  \cH\bra{ \rho | \gamma_{\lambda(\ell)}} &= \int_{\lambda(\ell)}^\eta \int_{\lambda(\ell)}^\xi \pderiv{M_1(\gamma_{\theta})}{\theta} \dx{\theta} \dx{\xi}
    = \int_{\lambda(\ell)}^\eta \int_{\lambda(\ell)}^\xi \Var(\gamma_{\theta}) \dx{\theta} \dx{\xi} .
  \end{align*}
  The result~\eqref{e:RelEnt:lambda} follows from the lower and upper bound on the variance~\eqref{def:cvar}.
  For the second identity, we consider in a similar manner the free energy difference $\cF(\rho) - \cF(\gamma_{\eta})$ and the relative entropy $\cH\bra*{\rho | \gamma_{\eta}}$
  \begin{align*}
    \MoveEqLeft{\cF(\rho) - \cF(\gamma_{\eta}) = \int \rho \log \rho + \int H \rho - \int \gamma_{\eta} \log \gamma_{\eta} - \int H \gamma_{\eta} }\\
    &= \int \rho \log \frac{\rho}{\gamma_{\eta}} - \int \rho \log Z_{\eta} + \eta \int x \rho + \int \gamma_{\eta} \log Z_{\eta} - \eta \int x \gamma_{\eta} \\
    &= \cH\bra*{\rho | \gamma_{\eta}} + \eta\bra{ \ell - M_1(\gamma_\eta) } .
  \end{align*}
  The estimate~\eqref{e:RelEnt:lambda:cor} is an immediate consequence of the bound~\eqref{e:lambda:deriv} on $\ell \mapsto \lambda(\ell)$, which implies
  \begin{equation}\label{e:lambda:biLipschitz}
     \frac{\abs{\ell^* - \ell}}{C_{\Var}} \leq  \abs*{\lambda(\ell^*) - \lambda(\ell)} \leq \frac{\abs{\ell^* - \ell}}{c_{\Var}} .
   \end{equation}
\end{proof}
The next Lemma calculates the entropy with respect to some time dependent parametrized steady states $\gamma_{\eta}$ for any $\eta\in C^1(\R_+,\R)$.
\begin{lem}\label{lem:quasistationaryEntDiss}
  Let $\eta \in C^1(\R_+,\R)$ and $\rho$ be a solution to~\eqref{e:cFP} then it holds
  \begin{equation}\label{e:quasistationaryEntDiss}
    \pderiv{}{t} \cH(\rho(t) | \gamma_{\eta(t)}) = - \cD(\rho(t),\sigma(t)) + \dot\ell(t)\bra*{ \sigma(t) - \eta(t)} - \dot\eta(t) \bra*{ \ell(t) - M_1(\gamma_\eta)} .
  \end{equation}
\end{lem}
\begin{proof}
  The proof consists in a straightforward calculation using the chain rule and several integrations by parts, where we for the sake of notation neglect the time dependence of $\rho$, $\eta$, $\sigma$ and $\ell$. We calculate
  \begin{align*}
    \pderiv{}{t} \cH\bra*{\rho | \gamma_{\eta}} &= \int \bra*{ \log\frac{\rho}{\gamma_\eta} + 1} \partial_t \rho - \dot\eta \int \rho \partial_\eta \log \gamma_\eta \\
    &= - \int \bra*{\partial_x \log \frac{\rho}{\gamma_\eta}} \; \bra*{\partial_x \log \frac{\rho}{\gamma_\sigma}} \; \rho - \dot \eta \int \bra*{x - \partial_\eta \log Z_\eta} \rho \\
    &= -\cD(\rho,\sigma) - \int \bra*{\partial_x \log \frac{\gamma_{\sigma}}{\gamma_{\eta}}} \; \bra*{\partial_x \log \frac{\rho}{\gamma_{\sigma}}} \; \rho - \dot\eta \bra*{ \ell - M_1(\gamma_\eta)} .
  \end{align*}
  The conclusion follows from
  \begin{align*}
    \int \bra*{\partial_x \log \frac{\gamma_{\sigma}}{\gamma_{\eta}}} \; \bra*{\partial_x \log \frac{\rho}{\gamma_{\sigma}}} \; \rho &= \bra*{\sigma - \eta} \int \bra*{ H' -\sigma} \rho = - \bra*{\sigma - \eta} \; \dot\ell ,
  \end{align*}
  by the equation of the Lagrange multiplier~\eqref{lagranian_multiplier_in_introduction}.
\end{proof}
The identity~\eqref{e:quasistationaryEntDiss} shows that the choice $\eta(t) = \lambda(\ell(t))$ has the advantage that the term involving $\dot\eta$ vanishes.
With this preliminary considerations, we can prove the quantitative long-time behaviour.
\begin{thm}\label{thm:LT:quant}
 Let Assumption~\ref{assume:LT:qual} be satisfied, then for a solution to~\eqref{e:cFP} with constraint~\eqref{lagrange_multiplier} holds
 \begin{equation}\label{e:LT:quant:RelEnt:eta:l2}
   \cH\bra*{\rho(t) | \gamma_{\lambda(\ell(t))}} \leq e^{-\tau t}  \cH\bra*{\rho | \gamma_{\lambda(\ell(0))}}  +  C_{\ell,\sigma} \int_0^t e^{-\tau (t-s)} \abs[\big]{\dot \ell(s)} \dx{s} ,
 \end{equation}
 where $\tau := C_{\LSI,\norm{\sigma}_{L^\infty}}^{-1} > 0$ from Lemma~\ref{lem:LSI} and $C_{\ell,\sigma} :=\norm*{\abs{\lambda(\ell(\cdot))}}_{L^\infty(\R_+)}  + \norm*{\sigma(\cdot)}_{L^\infty(\R_+)} < \infty$. In particular, if $\ell$ satisfies for some $\kappa>0$ and $L_0>0$ and all $t\geq 0$
  \begin{equation}\label{assume:LT:ell:quant}
    \abs{\dot\ell(t)} \leq L_0 e^{-\kappa t},
  \end{equation}
  it holds
  \begin{equation}\label{e:LT:quant:RelEnt:eta}
    \cH\bra*{\rho(t) | \gamma_{\lambda(\ell(t))}} \leq e^{-\tau t} \cH\bra*{\rho | \gamma_{\lambda(\ell(0))}} + C_{\ell,\sigma} \, L_0 \; \frac{e^{-\tau t} - e^{-\kappa t}}{\kappa - \tau}.
  \end{equation}
\end{thm}
\begin{rem}
  The exponential term in~\eqref{e:LT:quant:RelEnt:eta} is bounded by
  \begin{equation}
     \frac{e^{-\tau t} - e^{-\kappa t}}{\kappa - \tau} \leq  \begin{cases}
      \frac{e^{-\tau t}}{\kappa - \tau} &, \kappa > \tau \\
      t e^{-\kappa t} &, \kappa = \tau \\
      \frac{e^{-\kappa t}}{\tau - \kappa} &, \kappa < \tau .
    \end{cases}
  \end{equation}
  Hence, whenever $\tau < \kappa$, it holds for an explicit $C=C(\tau, \kappa, L_0 , C_{\ell,\sigma})$ the bound
  \begin{equation}
    \cH\bra*{\rho(t) | \gamma_{\lambda(\ell(t))}} \leq e^{-\tau t} \bra*{ \cH\bra*{\rho | \gamma_{\lambda(\ell(0))}} + C} .
  \end{equation}
 \end{rem}
\begin{proof}
  We have from~\eqref{e:bound:M2sigma} in Lemma~\ref{lem:bound:M2sigma} that $\norm{\sigma}_{L^\infty(\R_+)} < \infty$, which by Lemma~\ref{lem:LSI} implies that there is a uniform constant $\tau := C_{\LSI,\norm{\sigma}_{L^\infty}}^{-1}$ such that $\tau \, \cH(\rho| \gamma_{\sigma}) \leq \cD(\rho,\sigma)$ for all $\abs{\sigma }\leq M$. Moreover, the bi-Lipschitz estimate~\eqref{e:lambda:deriv} on $\lambda$ and since $\ell \in L^\infty(\R_+)$ by Assumption~\ref{assume:LT:qual}, we also get $\norm{\lambda(\ell(\cdot))}_{L^\infty(\R_+)} < \infty$. Hence, the constant $C_{\ell,\sigma} <\infty$ is well-defined.

  \noindent With these preliminary considerations we apply Lemma~\ref{lem:quasistationaryEntDiss} with the choice $\eta(t) = \lambda(\ell(t))$ with $\lambda$ defined implicitly in~\eqref{def:lambda} of Proposition~\ref{prop:constraint:energy:min} and note that the last term in~\eqref{e:quasistationaryEntDiss} vanishes, since $\ell(t) = \lambda\bra*{ \gamma_{\lambda(t)}}$ by definition. By neglecting for brevity the explicit time dependence in the notation, we calculate
  \begin{align*}
    \pderiv{}{t} \cH(\rho | \gamma_{\lambda(\ell}) &= - \cD(\rho,\sigma) + \dot\ell\bra*{ \sigma - \lambda(\ell) } \\
    &\stackrel{\mathclap{\eqref{def:LSI}}}{\leq} - \tau \cH\bra*{\rho | \gamma_{\sigma} } + \dot\ell\bra*{ \sigma - \lambda(\ell)} \\
   &\stackrel{\mathclap{\eqref{e:RelEnt:lambda}}}{\leq} - \tau \cH\bra*{\rho | \gamma_{\lambda(\ell)}}
   + C_{\ell,\sigma} \abs{\dot\ell}.
  \end{align*}
  Integrating this equation gives~\eqref{e:LT:quant:RelEnt:eta:l2}, which after using~\eqref{assume:LT:ell:quant} yields~\eqref{e:LT:quant:RelEnt:eta}.
\end{proof}
There are several possible reformulations of the statement. First, we investigate the convergence of the free energy difference $\cF(\rho) - \cF(\gamma_{\sigma^*})$, which is related to the relative entropy $\cH\bra*{\rho | \gamma_{\sigma^*}}$.
\begin{cor}\label{cor:LT:RelEnt*}
  Under the exponential convergence Assumption~\eqref{assume:LT:ell:quant} on $\ell$, it follows
  \begin{equation}\label{e:LT:RelEntsigma*}
    \cH\bra{\rho(t)| \gamma_{\sigma^*}} \leq e^{-\tau t}  \cF(\rho^0) + L_0 \, C_{\ell,\sigma} \ \frac{e^{-\tau t} - e^{-\kappa t}}{\kappa - \tau}  + \frac{C_{\Var}L_0^2}{2 c_{\Var}^2 \kappa^2} \  e^{-2\kappa t}
  \end{equation}
  as well as
  \begin{equation}\label{e:LT:FreeEnergy}
    \abs[\big]{ \cF(\rho(t)) - \cF(\gamma_{\sigma^*}) - \cH\bra*{\rho(t) | \gamma_{\sigma^*}}} \leq \frac{\abs{\sigma^*} L_0}{\kappa} e^{-\kappa t}
  \end{equation}
\end{cor}
\begin{rem}
  The convergence in relative entropy also implies convergence the $L^1$-norm via the classical Csiszár-Kullback-Pinsker inequality: For all $\rho,\gamma \in \cP(\R)$ holds
\begin{equation}\label{e:CKP}
  \int \abs{\rho(x) - \gamma(x)} \dx{x} \leq \sqrt{2 \cH(\rho | \gamma)}.
\end{equation}
  Hence, we have for instance in the case of $\tau < \kappa$
  \begin{equation}
    \int \abs*{\rho(t,x) - \gamma_{\sigma^*}(x)}\dx{x} \leq \sqrt{2 \cH\bra{\rho(t)| \gamma_{\sigma^*}}} \leq  C e^{-\frac{\tau}{2} t} ,
  \end{equation}
  for some $C$ explicitly given in terms of the constants on the right hand side of~\eqref{e:LT:RelEntsigma*}.
\end{rem}
\begin{proof}
  Let us start by using the comparison~\eqref{e:RelEnt:lambda:cor} to estimate $\cH\bra*{\rho | \gamma_{\lambda(\ell)}}$ from above by $\cH\bra*{\rho | \gamma_{\sigma^*}}$. Then, we can conclude~\eqref{e:LT:RelEntsigma*} from the convergence assumption~\eqref{assume:LT:ell:quant} implying
  \begin{equation*}
    \abs{\ell^* - \ell(t)} \leq \frac{L_0}{\kappa} e^{-\kappa t} .
  \end{equation*}
  On the other hand, we can use the comparison~\eqref{e:FreeEnergy:RelEnt:Ident} for free energy difference $\cF(\rho) - \cF(\gamma_{\sigma^*})$ and the relative entropy relative entropy $\cH\bra*{\rho | \gamma_{\sigma^*}}$, which yields the identity
  \begin{align*}
    \abs*{\cF(\rho) - \cF(\gamma_{\sigma^*}) - \cH\bra*{\rho | \gamma_{\sigma^*}}} = \abs*{\sigma^*} \abs*{ \ell - \ell^* } .
  \end{align*}
  Hence, for all $t$ and by the convergence assumption~\eqref{assume:LT:ell:quant} on $\ell$ follows~\eqref{e:LT:FreeEnergy}.
\end{proof}
To obtain a convergence statement for the Lagrange multiplier, we introduce the weighted Csiszár-Kullback-Pinsker inequality due to~\cite{BV05}.
\begin{lem}[{Weighted Csiszár-Kullback-Pinsker inequality \cite[Theorem 2.1]{BV05}}]\label{lem:Pinsker}$ $\\
  Let $\rho, \gamma \in \cP_2(\R_+)$ be absolutely continuous. Let $w:\R_+ \to \R_+$ be such that $C_w := \int e^{w^2} \dx{\gamma} < \infty$. Then, it holds
  \begin{equation}\label{e:Pinsker:sub}
    \int w \abs{\rho - \gamma} \leq \sqrt{2\bra*{ 1+ \log C_w} \cH\bra*{\rho | \gamma}} .
  \end{equation}
\end{lem}
The weighted Csiszár-Kullback-Pinsker inequality allows to compare $\bra{\sigma-\sigma^*}$ with the relative entropy $\cH\bra*{\rho | \gamma_{\lambda(\ell)}}$.
\begin{cor}\label{cor:LT:sigma*}
  Under Assumption~\ref{assume:LT:qual} holds for all $t\geq 0$ and for a constant $C$ depending on the potential $H$ and the constants in Assumption~\ref{assume:LT:qual} the estimate
  \begin{equation}
    \bra*{\sigma(t)-\sigma^*}^2 \leq C \cH\bra*{\rho(t) | \gamma_{\lambda(\ell(t))}} + 4 \abs[\big]{\dot \ell(t)}^2 + \frac{2}{c_{\Var}^{2}} \abs{ \ell(t) -\ell^*}^2 .
  \end{equation}
  In particular, under the exponential convergence Assumption~\eqref{assume:LT:ell:quant} on $\ell$ with $\kappa > \tau$ it follows for some explicit constant $\tilde C$
  \begin{equation}
    \bra*{\sigma(t) - \sigma^*}^2 \leq \tilde C e^{-\tau \, t}.
  \end{equation}
\end{cor}
\begin{proof}
  It is sufficient to estimate the difference $\bra{\sigma - \lambda(\ell)}^2$ and $\bra{\lambda(\ell) - \sigma^*}$. Let us note, that for all $\eta\in \R$ it holds the identity
  \begin{equation*}
    0 = \int \partial_x \gamma_\eta(x) \dx{x} = \int \bra*{ H' - \eta} \dx{\gamma_{\eta}} .
  \end{equation*}
  Hence, $\eta = \int H' \dx{\gamma_\eta}$ and together with the definition~\eqref{lagranian_multiplier_in_introduction} of $\sigma$ follows that
  \begin{align*}
    \bra*{ \sigma -\lambda(\ell)}^2 \leq 2\bra*{ \int \abs{ H'} \, \abs*{ \rho - \gamma_{\lambda(\ell)}}}^2 + 2\abs[\big]{\dot \ell}^2.
  \end{align*}
  Assumption~\ref{assume:LT:qual} implies that there exists $C_H$ such that $\abs{H'(x)} \leq C_H\bra{1+\abs{x}}$ for all $x\in \R$. By choosing $w := \min\set{c_{H,+},c_{H,-}}/2 \; (1+\abs{x})$ with $c_{H,\pm}$ from Assumption~\ref{assume:LT:qual}, we obtain since $\abs{\lambda(\ell)} \leq \lambda(0) + \norm{\ell}_{L^\infty} / c_{\Var} =: M$ that there exists $C_M$ such that $\int e^{w^2} \dx{\gamma_{\lambda(\ell)}} \leq C_M$. Hence, we can apply Lemma~\ref{lem:Pinsker} and obtain
  \begin{equation}
    \bra*{ \int \abs{ H'} \, \abs*{ \rho - \gamma_{\lambda(\ell)}}}^2 \leq \frac{8 C_H^2}{\min\set{c_{H,+}^2,c_{H,-}^2}}\bra{1+ \log C_M} \cH\bra*{ \rho | \gamma_{\lambda(\ell)}} .
  \end{equation}
  It remains to estimate $\bra{\lambda(\ell) - \sigma^*}$, which immediately follows from the bi-Lipschitz estimate~\eqref{e:lambda:biLipschitz} by noting that $\sigma^* = \lambda(\ell^*)$.
\end{proof}

\subsection{Rate of convergence in dependence on \texorpdfstring{$\nu$}{nu}}\label{s:LT:nu}

In this section, we outline how the results of the previous section can describe the behaviour of the system in different regimes. Let us include again the dependence on the
viscosity parameter~$\nu^2$ in the equation as stated in~\eqref{e:cFP} but keep the time scale $\tau=1$.

Let us incorporate the $\nu^2$ dependence in to the local equilibrium states by defining
\begin{equation}\label{def:gamma:nu}
  \gamma_{\sigma,\nu}(x) := \frac{1}{Z_{\sigma,\nu}} \exp\bra*{- \frac{H(x) - \sigma\, x}{\nu^2}}  \quad\text{with}\quad Z_{\sigma,\nu} := \int \exp\bra*{- \frac{H(x) - \sigma\, x}{\nu^2}}  \dx{x} .
\end{equation}
Moreover, let us describe the potential $H$ in more detail. The quadratic growth~\eqref{assume:LT:H:quadratic_growth} of Assumption~\eqref{assume:LT:qual} implies that the \emph{spinodal region} is finite, that is
\begin{equation}
  \Omega := \set{ x : H''(x) \leq 0 } \qquad\text{satisfies}\qquad \abs{\Omega} < \infty .
\end{equation}
Let us for the sake of simplicity of the presentation assume that $H$ does not have flat pieces, i.e. $\set{ x : H''(x) = 0}$ is a null set.

It is possible to deduce the scaling of the constants $c_{\Var,\nu}$ and $C_{\Var,\nu}$ defined analogously as in~\eqref{def:cvar} and it holds
\begin{equation}\label{def:cvar:nu}
  c_{\Var,\nu} = \inf_{\eta\in \R} \Var_{\gamma_{\eta,\nu}} \gtrsim \nu^2 \qquad\text{and}\qquad C_{\Var,\nu} = \sup_{\eta\in \R} \Var_{\gamma_{\eta,\nu}} \lesssim \nu^2 + \abs{\Omega} .
\end{equation}
Theorem~\ref{thm:LT:quant} shows that the exponential rate of convergence is determined by the logarithmic Sobolev constant as defined in~\eqref{def:LSI} and we need a $\nu$-dependent analysis for it. To do so, we introduce the set
\begin{equation}
  \Sigma = \set*{\sigma \in \R :  H'(\cdot ) = \sigma \text{ has more than one solution} } .
\end{equation}
The quadratic growth assumption~\eqref{assume:LT:H:quadratic_growth} on $H$ implies once more that $\abs{\Sigma}<\infty$. Hence, in the case $\sigma \in \Sigma$ the function $H_{\sigma}(x) = H(x) - \sigma x$ is a multiwell potential. That is, there is a finite set of points $\set{X_{i}(\sigma)}_{i=0}^M$ and we can w.l.o.g.\@ assume that $X_0(\sigma)$ is a global minimum. Then the potential posses an energy barrier defined as the largest energy difference between any local minima and the global minimum, that is for $\sigma \in \Sigma$
\begin{equation}
  \Delta H_\sigma := \!\!\!\!\!\! \inf_{\varphi\in C([0,1],\R)} \sup_{t\in [0,1]} \set*{ H_\sigma(\varphi(t)) - H_\sigma(\varphi(0)): \varphi(0) \in\set{X_{i}(\sigma)}_{i=1}^M, \varphi(1) = X_0(\sigma)} .
\end{equation}
Now, we can formulate the dependency of the logarithmic Sobolev constant on $\nu$.
\begin{lem}[Logarithmic Sobolev inequality ($\nu$-dependent)]\label{Lemmalogsob}
  Under Assumption~\ref{assume:LT:qual} satisfies the measure $\gamma_{\sigma,\nu}$ from~\eqref{def:gamma:nu}  for any $\sigma \in \R$ a logarithmic Sobolev inequality with constant $C_{\LSI,\nu}(\sigma)>0$. Here the constant $C_{\LSI,\nu}(\sigma)$ satisfies
  \begin{equation}
   C_{\LSI,\nu}(\sigma) \lesssim
  \begin{cases}
     \nu^{-2} &, \sigma \not\in \Sigma \\
     \nu^{-2}\exp\bra*{ \frac{\Delta H_\sigma}{\nu^2}} &, \sigma \in \Sigma .
  \end{cases}
  \end{equation}
  Moreover, if $H$ is convex with $H''(x) \geq k>0$ for all $x\in \R$, it holds
  \begin{equation}
     C_{\LSI,\nu}(\sigma) \leq \frac{1}{k} .
  \end{equation}
\end{lem}
\begin{proof}
  The results are contained in the literature. However, we have to be careful by translating the results and the scaling of $\nu$. In \cite{MS12}, the following logarithmic Sobolev inequality is proven
  \[
    \cH(\rho | \gamma_{\sigma,\nu}) \leq \tilde C_{\LSI,\nu}(\sigma) \int \bra*{ \partial_x  \log \frac{\rho}{\gamma_{\sigma,\nu}}}^2 \dx{\rho} .
  \]
  By comparing \eqref{e:def:freeEnergy} and~\eqref{e:def:RelEnt} it follows that $\nu^2 \cH(\rho|\gamma_{0,\nu^2}) = \cF(\rho)$. Likewise, the dissipation $\cD$ in~\eqref{e:def:Dissipation} comes with an additional factor $\nu^4$ and therefore the $\nu$-dependent logarithmic Sobolev inequality from~\eqref{def:LSI} takes the form
  \[
    \nu^2 \cH(\rho | \gamma_{\sigma,\nu}) \leq C_{\LSI,\nu}(\sigma) \nu^4 \int \bra*{ \partial_x  \log \frac{\rho}{\gamma_{\sigma,\nu}}}^2 \dx{\rho} .
  \]
  By comparison, we can read of $C_{\LSI,\nu}(\sigma) = \nu^{-2} \tilde C_{\LSI,\nu}(\sigma)$. The detailed scaling of $\tilde C_{\LSI,\nu}(\sigma)$ in $\nu$ for the case $\sigma\not\in \Sigma$ follows from \cite[Theorem 2.10]{MS12}. The case $\sigma\not\in \Sigma$ is content of \cite[Corollary 2.17]{MS12}. Finally, the convex case follows as in Lemma~\ref{lem:LSI}, by taking the rescaled Hamiltonian $H/\nu^2$ into account.
\end{proof}
With Lemma \ref{Lemmalogsob} we can now conclude on the different cases. For that, we assume that $\ell$ satisfies the exponential convergence Assumption~\ref{assume:LT:ell:quant} for some $\kappa$,
which in the next statements is always assumed to be larger then $\tau$. Then, there exists $\tau > 0$ and a constant $C_0$ only depending on the initial values and $L_0$ such that
\begin{equation}\label{e:bound:nu}
  \abs*{\sigma(t) - \sigma^*} + \int \abs*{ \rho(t) - \gamma_{\lambda(\ell^*)}} \leq C_0 \exp\bra*{ - \frac{\tau}{2} \, t} ,
\end{equation}
where $\tau = \tau(\nu)$ is given as follows

\noindent\emph{Convex case:} If $H''(x) \geq k > 0$, then $\tau_{\convex} = k$.

\noindent\emph{Unimodal case:} If $\abs{\Sigma} = 0$, then $\tau_{\unimodal} = c \nu^2 $ for some $c>0$.

\noindent\emph{Kramers case:} If $\abs{\Sigma} > 0$, then $\tau_{\Kramers} = c \nu^2 \exp\bra*{ - \frac{ \Delta H^*}{\nu^2} }$ for some $c>0$ and the $\Delta H^* := \sup_{\sigma \in \Sigma} \Delta H_\sigma$.

Let us point out that in Kramers case, the convergence rate can be improved by also taking $\sigma^* = \lambda(\ell^*)$ into the account. The bound~\eqref{e:bound:nu} allows for a self-improvement of $\tau_{\Kramers}$ in time in the case when $\sigma^* \not\in \Sigma$. By defining the time $T_0(\sigma^*,\nu)$ such that $C_0 \exp\bra*{ - \frac{\tau_{\Kramers}}{2}\, T_0(\sigma^*,\nu)} = \operatorname{dist}(\sigma^*, \Sigma)=: d_*$, it follows from~\eqref{e:bound:nu} that $\sigma(t) \not\in \Sigma$ for all $t\geq T_0(\sigma^*,\nu)$ and hence, we are back in the unimodal case. The total convergence estimate becomes
\begin{equation}
  \abs*{\sigma(t) - \sigma^*} + \int \abs*{ \rho(t) - \gamma_{\lambda(\ell^*)}} \leq d_0 \bra*{\max\set*{\frac{C_0}{d_0},1}}^{\frac{\tau_{\unimodal}}{\tau_{\Kramers}}} \exp\bra*{ - \frac{\tau_{\unimodal}}{2} t} .
\end{equation}
The estimate shows, that at least in the case of sufficiently well-prepared initial values such that $C_0 \leq \operatorname{dist}(\sigma^*, \Sigma)$, the convergence rate is not exponentially small but behaves linear in $\nu^2$. The reason is, that the a priori estimates on the Lagrange multiplier does ensure, that the effective potential $H_\sigma$ has always a unimodal structure.

\subsection*{Acknowledgment}
B.N.\ and A.S.\ acknowledge support through the CRC 1060
\emph{The Mathematics of Emergent Effects} at the University of Bonn that is funded through the German Science Foundation (DFG).

\bibliographystyle{abbrv}
\bibliography{bib}

\begin{thebibliography}{10}
\bibitem{Luckhaus1990}
S.~Luckhaus.
\newblock {Solutions for the two-phase Stefan problem with the Gibbs-Thomson
  law for the melting temperature.}
\newblock {\em Eur. J. Appl. Math.}, 1:101--111, 1990.
\end{thebibliography}

\end{document}